\documentclass[oneside, 11pt]{amsart}
\usepackage{amsmath,amssymb,amsthm}
\usepackage[dvips]{graphicx}
\usepackage{color}              % Need the color package
\usepackage{epsfig}

\usepackage{tikz}
\usetikzlibrary{arrows,decorations.pathmorphing, decorations.pathreplacing,backgrounds,positioning,fit}

\newtheorem{theorem}{Theorem}[section]
\newtheorem{proposition}[theorem]{Proposition}
\newtheorem{corollary}[theorem]{Corollary}
\newtheorem{lemma}[theorem]{Lemma}

\newtheorem{introthm}{Theorem}
\newtheorem{introcor}[introthm]{Corollary}

\theoremstyle{definition}
\newtheorem{definition}[theorem]{Definition}

\theoremstyle{remark}
\newtheorem{remark}[theorem]{Remark}

\newcommand{\thmref}[1]{Theorem~\ref{Thm:#1}}
\newcommand{\propref}[1]{Proposition~\ref{Prop:#1}}
\newcommand{\secref}[1]{\S\ref{Sec:#1}}
\newcommand{\lemref}[1]{Lemma~\ref{Lem:#1}}
\newcommand{\corref}[1]{Corollary~\ref{Cor:#1}}
\newcommand{\figref}[1]{Fig.~\ref{Fig:#1}}

\newcommand{\eqnref}[1]{Equation~\eqref{Eq:#1}}
\newcommand{\defref}[1]{Definition~\ref{Def:#1}}

\newcommand{\calA}{{\mathcal A}}
\newcommand{\calB}{{\mathcal B}}
\newcommand{\calC}{{\mathcal C}}
\newcommand{\calD}{{\mathcal D}}
\newcommand{\calE}{{\mathcal E}}
\newcommand{\calF}{{\mathcal F}}
\newcommand{\calG}{{\mathcal G}}

\newcommand{\M}{{\mathcal M}}
\newcommand{\calM}{{\mathcal M}}
\newcommand{\calN}{{\mathcal N}}

\newcommand{\calP}{{\mathcal P}}

\newcommand{\T}{{\mathcal T}}

\newcommand{\calV}{{\mathcal V}}
\newcommand{\calW}{{\mathcal W}}
\newcommand{\calX}{{\mathcal X}}
\newcommand{\calY}{{\mathcal Y}}
\newcommand{\calZ}{{\mathcal Z}}

\newcommand{\NE}{\calN\calE}

\newcommand{\HH}{{\mathbb H}}
\newcommand{\R}{{\mathbb R}}
\newcommand{\ZZ}{{\mathbb Z}}

  \newcommand{\gothic}{\mathfrak}
  \newcommand{\ga}{{\gothic a}}
  \newcommand{\gb}{{\gothic b}}
  \newcommand{\gc}{{\gothic c}}
  
\newcommand{\from}{\colon \thinspace}
\newcommand{\ST}{{\: \Big| \:}}

\renewcommand{\d}{{\rm d}}

\newcommand{\emul}{\stackrel{{}_\ast}{\asymp}}
\newcommand{\gmul}{\stackrel{{}_\ast}{\succ}}
\newcommand{\lmul}{\stackrel{{}_\ast}{\prec}}
\newcommand{\eadd}{\stackrel{{}_+}{\asymp}}
\newcommand{\gadd}{\stackrel{{}_+}{\succ}}
\newcommand{\ladd}{\stackrel{{}_+}{\prec}}

\newcommand{\ep}{\epsilon}

\newcommand{\Teich}{Teich\-m\"uller~}

\newcommand{\param}{{\mathchoice{\mkern1mu\mbox{\raise2.2pt\hbox{$
\centerdot$}}
\mkern1mu}{\mkern1mu\mbox{\raise2.2pt\hbox{$\centerdot$}}\mkern1mu}{
\mkern1.5mu\centerdot\mkern1.5mu}{\mkern1.5mu\centerdot\mkern1.5mu}}}

\DeclareMathOperator{\twist}{twist}

\DeclareMathOperator{\Mod}{Mod}

\DeclareMathOperator{\I}{i}

\DeclareMathOperator{\rank}{{rank}}

\DeclareMathOperator{\Map}{{Mod}}

\DeclareMathOperator{\Vol}{Vol}

\newcommand{\bDelta}{\overline \Delta}  
\newcommand{\balpha}{{\boldsymbol \alpha}}

\newcommand{\AM}{{\calA \calM}}
\newcommand{\C}{{\calC}}

\newcommand{\X}{{\calX}}

\newcommand{\Gr}{{\text{G}}}

\newcommand{\bY}{\mathbf{Y}}
\newcommand{\bC}{\mathbf{C}}
\newcommand{\bP}{\mathbf{P}}

\begin{document}

\title{Large scale rank of Teichm\"uller space}
\author{Alex Eskin}
\author{Howard Masur}
\author{Kasra Rafi}

\thanks{\tiny \noindent 
The first author is partially supported by NSF grants  DMS 0905912 and 
DMS 1201422. The second author is partially supported by NSF
grant DMS-0905907.  The third author is partially supported by
by NSF grant DMS-1007811.}

\date{\today}

\begin{abstract}
Let $\X$ be quasi-isometric to either the mapping class group equipped
with the word metric, or to \Teich space equipped with either the \Teich metric 
or the Weil-Petersson metric. We introduce a unified approach to study the coarse 
geometry of these spaces. We show that the quasi-Lipschitz image in $\X$ of 
a box  in $\R^n$ is locally near a standard model of a flat in $\X$. As a 
consequence, we show that, for all these spaces, the geometric rank and 
the topological rank are equal. The methods are axiomatic and apply to 
a larger class of metric spaces.
\end{abstract} 

\maketitle

\section{Introduction}

In this paper we study the large scale geometry of several metric spaces:
the \Teich space $\T(S)$ equipped with the \Teich metric $d_\T$, the \Teich space
equipped with the Weil-Petersson metric $d_{W\!P}$ and the mapping class 
group $\Mod(S)$ equipped with the word metric $d_W$. (Brock \cite{brock:WV} showed that the Weil-Petersson metric is quasi-isometric to the pants complex). Even though 
the definitions of distance in these spaces are very different, they share a key 
feature, namely, an inductive structure. That is, they are a union of product regions 
associated to lower complexity surfaces with the gluing pattern given by the 
curve complex. 

Let $S$ be a possibly disconnected surface of finite hyperbolic type and
let $\X=\X(S)$ be a metric space that is quasi-isometric to one of the metric 
spaces mentioned above. One major goal in understanding the large scale geometry of a metric space is prove quasi-isometric rigidity of the space; that any quasi-isometry is bounded distance from an isometry. The usual starting point is to understand its flats. By a flat here we mean a quasi-isometric image of Euclidean space. We analyze 
quasi-Lipschitz maps from a \emph{large box} $B\subset \R^n$ into $\X$.  Our 
goal is to give a  description of the image of such a map on a \emph{large sub-box}
$B'\subset B$. We show that the image of $B'$ looks like a \emph{standard flat} 
up to a small linear error. A standard flat is a  product of {\em preferred paths} 
associated to disjoint subsurfaces of $S$
(see \defref{Preferred} and \defref{StandardFlat}). 

Our main theorem is the following. 
\begin{introthm}[Image of a box is locally standard] 
\label{Thm:Intro-Standard-Flat} 
Let $\X$ be either of \Teich space $\T(S)$ equipped with the \Teich metric $d_\T$,  \Teich space
with the Weil-Petersson metric $d_{W\!P}$, (or pants graph) or  the mapping class 
group $\Mod(S)$ equipped with the word metric $d_W$. For all $K,C$ and for all $R_0,\epsilon_0$ there exists $R_1$  such that if $B$ is 
a box of size at least $R_1$ and $f\from B\to \X$ is a $(K,C)$--quasi-Lipschitz 
map,  then  there is a sub-box $B'\subset B$ of size $R'\geq R_0$ such 
that  $f(B')$ lies inside an $O(\epsilon_0 R')$--neighborhood of a standard flat in 
$\X$.
\end{introthm}  

As a corollary, we determine the large scale rank of the space $\X$.
Define the \emph{topological rank} of $\X$, $\rank_{top}(\X)$, to be largest 
integer $r$ so that there are pairwise disjoint \emph{essential} subsurfaces 
$W_1, \ldots, W_r$ in $S$.  (In all cases,  a thrice-punctured
sphere $W$ is considered inessential. Also, as we shall see, when $\X$ is 
quasi-isometric to $(\T(S), d_{W\!p})$, annuli are also considered 
inessential.)  The dimension of a standard flat in $\X$ is at most  
 $\rank_{top}(\X)$.

\begin{introthm}[Geometric rank] \label{Thm:Rank}
For every $K$ and $C$, there is a constant $R_2$ so that if $B$ is a box of
size at least $R_2$ in $\R^n$ and $f \from B \to \X$ is a $(K,C)$--quasi-isometric
embedding, then $n \leq \rank_{top}(\X)$. Furthermore, for $n=\rank_{top}(\X)$, 
there is a quasi-isometric embedding of a Euclidean $n$ dimensional half space  into $\X$. 
\end{introthm}

Define the \emph{geometric rank} of $\X$, $\rank_{geo}(\X)$, to be largest 
integer $n$ so that there exists $K,C$ such that for any $R$, there is a $(K,C)$ quasi-isometric embedding 
$f$ of a ball $B\subset\R^n$ of radius $R$ into $\X$. Also, let ${\sf g}$ be the genus of $S$, 
${\sf p}$ be the number of punctures of $S$ and ${\sf c}$ be the number
of component of $S$.

\begin{introcor} \label{Cor:Top=Geo}
The topological and the geometric rank of $\X$ are equal. Namely,
if $\X$ is either $(\Map(S), d_W)$ or $\T(S), d_\T)$ then
\[
\rank_{geo}(\X) = \rank_{top}(\X) = 3{\sf g} + {\sf p} - 3 {\sf c},
\]
and if $\X$ is $(\T(S), d_{W\!P})$, then 
\[
\rank_{geo}(\X) = \rank_{top}(\X) = 
\left\lfloor \frac{3{\sf g} + {\sf p} - 2{\sf c}}{2} \right\rfloor.
\]
\end{introcor} 

\begin{remark}  In the case of $\Map(S)$ there are quasi-isometric embeddings of all of Euclidean space of dimension, the geometric rank into $\Map(S)$. In Theorem 1.3 of \cite{Bo} B.Bowditch proves that there is a quasi-isometric embedding of a Euclidean $n$-dimensional half-space into  $(\T(S),d_\T)$ if and only if $n\leq \rank_{top}(\T(S),d_\T)$.  In Theorem 1.4 he shows that there are quasi-isometric embeddings of the {\em entire} Euclidean space of dimension $\rank_{top}(\T(S),d_\T)$ into Teichm\"uller space if and only if the genus of  $S$ is at most $1$ or $S$ is a closed surface of genus $2$.   We are grateful to him for pointing out an error in a previous version of this paper.
\end{remark}
 
\begin{remark} \label{Rem:More-Complexes} 
\thmref{Intro-Standard-Flat} and \thmref{Rank}
hold for a larger class of metric  spaces than are discussed above. Essentially,
one needs a mapping class group action and a distance formula similar to 
\eqnref{Distance} (see Masur-Schleimer \cite{masur:DC} for examples of such distance 
formulas). The definition of an essential surface has to be modified to mean any 
type of surface that appears is the associated distance formula. 

For many such spaces, e.g, the arc complex and the disk complex are known 
to be Gromov hyperbolic and therefore have geometric rank one 
\cite{masur:DC}, hence the corollary is already known.  Others, such as the 
Hatcher-Thurston complex and the separating curve complex, 
are not Gromov hyperbolic and our discussion applies to prove the
geometric rank and topological rank are equal. These complexes have been  
omitted to simplify the exposition. 
\end{remark}

\subsection*{History}
The idea of studying the rank of these objects was introduced by Brock-Farb 
\cite{brock:CR}. In the case when $\X$ is the pants graph \corref{Top=Geo} 
was first proven in that paper when the surface is the twice punctured torus.  
They also showed that the topological rank is always at most as large as the 
geometric rank and  conjectured \corref{Top=Geo} for all genera. \corref{Top=Geo}
was then proven for all genera in the case when $\X$ is quasi-isometric to the 
mapping class group with the word metric or \Teich space with the Weil-Petersson 
metric by Behrstock-Minsky and Hamenst\"adt \cite{minsky:DR, ham:QR}. 
The rank statement is used to prove the quasi-isometric rigidity of $\Mod(S)$
by Behrstock-Kleiner-Minsky-Mosher in \cite{minsky:MR} and by Hamenst\"adt
in \cite{ham:QR}.
The case of  \Teich space with the \Teich metric had not been studied previously.  
Note that the map that sends $\Mod(S)$ to the orbit of a point in $\T(S)$  is 
not a quasi-isometry or even a quasi-isometric embedding because of the thin regions  in $\T(S)$ which locally look like products of horoballs. 
Unlike \cite{minsky:DR} which uses  asymptotic cones, our approach as outlined below is to study the local behavior  of a quasi-Lipschitz maps. 

\subsection*{Main tools}
To prove our theorems we develop further  some tools that already exist
in the literature. The first one is the  idea of coarse differentiation.
This was introduced in the context of geometric group theory by Eskin-Fisher-Whyte \cite{eskin:CDI, eskin:CDII} (see references in that paper for its use in other contexts)  and  
used to prove quasi-isometric rigidity of lattices in Sol and in the quasi-isometry 
classification of lamplighter groups. The statement they used is similar to 
\thmref{Coarse} below which holds for quasi-Lipschitz maps between more general 
metric spaces. However, since we are mostly concerned with maps where the 
domain is a subset of $\R^n$, we prove the following statement which is cleaner 
and easier to use. 
 
\begin{introthm}[Coarse Differentiation] \label{Thm:Diff}
For every $\ep_0$ and $R_0$ there is $R_1$ so that for $R \geq R_1$ 
the following holds. Let $f \from B \to \calY$ be a quasi-Lipschitz map 
where $B$ is a box of size $R$ in $\R^n$. Then, there is a box 
$B' \subset B$ of size $R'\gmul R_0$ so that $f$ restricted to $B'$ is 
$\ep_0$--efficient on scale $R'$. 
\end{introthm} 

Here, efficient means that the image of  every line in the box  satisfies a reverse triangle inequality  
up to a small multiplicative error. That is; lines are mapped to lines up to a
\emph{sub-linear error}. One should think of the above theorem
as a coarse version of the Rademacher's theorem that 
%\begin{introthm}[Rademacher] \label{Thm:Rad} 
if  $f:\R^n\rightarrow \R^m$ is lipschitz, then $f$ is 
differentiable almost everywhere.
In \thmref{Diff} the sub-box $B'$ is an analogue of a point of differentiability. 

The importance of efficiency lies first of all in the fact that in a product space the projection of an efficient path to a factor is still efficient. The corresponding statement for quasi-geodesics is false.  
Furthermore efficient maps into hyperbolic 
spaces are easy to control (\propref{Efficient-to-Hyperbolic}). 
We also use the construction of Bestvina-Bromberg-Fujiwara 
\cite{bestvina:GQ}. They embed the mapping 
class group into a product of finitely many hyperbolic spaces. 
Their  construction is axiomatic and can be adopted easily to embed
any of our spaces $\X$ into a product of finitely many hyperbolic spaces. 
The $L^1$--metric on this product induces a metric on the space
$\X$ and this is the metric with respect to which we apply the coarse
differentiation theorem. Note that, the notion of efficiency is not preserved
under quasi-isometry and the choice of metric here is essential. 
The conclusion of this discussion, Theorem~\ref{Thm:Black-Box}  will be that efficient paths  fellow travel paths with the same endpoints that have nice properties.  These latter paths which we call preferred paths  will play the role of geodesics.

The power then of  Theorem~\ref{Thm:Diff}  lies in the fact that one can add the assumption
of efficiency for free, just by replacing $B$ with a sub-box $B'$. Altogether this will mean on large boxes the image of every line fellow travels a preferred path.

Finally, we use the realization theorem of Behrstock-Minsky-Mosher-Kleiner
\cite{minsky:MR}. They provide a description of the image of 
the mapping class group in the product of curve complexes. 
We adopt it to provide the description of the image of $\X$. 
This is necessary to translate back the information obtained in each 
hyperbolic factor to information in $\X$. 

\subsection*{Outline of the paper}

Section~\ref{Sec:Coarse} is devoted to the development of coarse differentiation
theory and to the discussion of efficient maps. The main result is 
\thmref{Differentiable} as discussed above. We also establish the basic
properties of efficient maps and prove that efficient paths in Gromov hyperbolic 
spaces stay close to geodesics; an analogue of the Morse Lemma. 

In \secref{Model}, we discuss the combinatorial model for each of the spaces 
considered in the paper. The three seemingly different metric spaces above 
have very similar models. Namely, \Teich space equipped with the Weil-Petersson 
metric is quasi-isometric to the pants graph \cite{brock:WV}. The mapping class
group is quasi-isometric to the marking graph \cite{minsky:CCII} by work of Masur-Minsky 
and \Teich space equipped with the \Teich metric is quasi-isometric
to the space of augmented markings by work of Rafi and Durham. \cite{rafi:HT, durham:AM}. 

The advantage of this approach is that we can measure relative 
complexity of two points $x,y \in \X$ from the point of view of 
a subsurface $W$. This is the distance in the curve complex of $W$ between
the projections of $x$ and $y$ to $W$. The curve complex of every
surface $W$ is known to be Gromov hyperbolic. We then define a
coarse metric on each of these combinatorial models using 
a distance formula which is the sum over relative complexity 
from the point of view of different subsurfaces. Since we work in 
the category of spaces up to quasi-isometries, the distance needs to be defined
only up a multiplicative error. 

In Section~\ref{Sec:Efficient} we introduce the notion of preferred paths.
These are paths whose projections to every curve complex is a quasi-geodesic
and they replace the notion of geodesics in our spaces. 
The main statement in the section is \thmref{Black-Box} 
which shows that an efficient path stays near a preferred path joining its 
endpoints. Hence, the outcome of the Coarse Differentiation Theorem is
indeed a box where straight lines are mapped to straight lines up 
to the first order. This is the key tool for the rest of the paper.
The proof uses the construction in \cite{bestvina:GQ} which allows one to 
embed $\X$ into a product of hyperbolic spaces.  The projection of the 
efficient paths into each factor stays near a geodesic in that factor. We then
use this and consistency theorem (\thmref{Consistency}) 
to build the preferred path in $\X$ tracing the given efficient path. 

Section~\ref{Sec:Preferredprop} establishes some properties of preferred paths. 
The main ones are fellow traveling properties that say that under certain 
conditions, preferred paths that begin and end near the same point fellow travel 
in the middle (\propref{fellowtraveling1} and \propref{fellowtraveling2}).
These statements are used in the succeeding sections to build big 
boxes with the required properties.  In \secref{Local} the main inductive
step is proven (\thmref{Homework}) and in \secref{Proof} we assemble the 
proofs of the main theorems.

\subsection*{Treatment of constants} 
Suppose that $\calY$ and $\calZ$ are geodesic metric spaces. We say 
a map $f \from \calY \to \calZ$ is quasi-Lipschitz if there are
constants $K$ and $C$ so that 
$$
 d_\calZ \big( f(x_1),f(x_2) \big) \le K d_\calY(x_1,x_2) +C.
$$
We fix constants $K$ and $C$ once and for all.
We also fix an upper-bound for the complexity of the surface $S$
and the dimension $n$. When we say a constant is uniform, we mean
its value depends only on $K$, $C$, the topology of $S$ and the value of $n$
only. Similarly, we will use terms like quasi-isometric embedding or
quasi-isometry to mean that the associated constants are the same
as $K$ and $C$ fixed above. 

To simplify presentation, we try to avoid naming uniform constants whenever
possible. Instead, we adopt 
the following notations. Let $\ga$ and $\gb$ represent various quantities 
and let $M$ and $C'$ be uniform constants. We say \emph{$\ga$ is less than $\gb$ 
up to a multiplicative error}, $\ga \lmul \gb$, if $\ga \le M \, \gb$. 
We say $\ga$ and $\gb$ are \emph{comparable},
$\ga \emul \gb$, if we have both $\ga \lmul \gb$ 
and $\gb \lmul \ga$.
 
Using the similar notation when the error is additive or both additive
and multiplicative,  we say $\ga \ladd \gb$ if $\ga \le \gb+C'$
and $\ga \prec \gb$ if $\ga \leq M \ga + C'$. Again, $\ga \eadd \gb$ if we have 
both $\ga \ladd \gb$ and $\gb \ladd \ga$ and $\ga \asymp \gb$ if we have both 
$\ga \prec \gb$ and $\gb \prec \ga$. Also, we often use the notation 
$\ga = O(\gb)$ to mean $\ga \lmul \gb$. For example
$$
\ga \ladd \gb + O(\gc)  \quad \Longleftrightarrow \quad
\ga \leq \gb + M \gc + C',
$$
for uniform constants $M$ and $C'$. 

Using this notation we may write 
\[
\ga \gadd \gb \quad \text{and} \quad \gb \gadd \gc
\quad\Longrightarrow \quad \ga \gadd \gc. 
\]
Here, the additive error in the last inequality is the sum of the additive 
errors in the first two inequalities and hence is still a uniform constant. 
That is, different occurrences of $\gadd$ have different implied constants. 
But as long as we use statements of this type a uniformly bounded number of
times, all the implied constant are still uniform.

\section{Coarse Differentiation} \label{Sec:Coarse}
Being differentiable means that, to first order, lines are mapped to lines
and points along a line satisfy the reserves triangle inequality. 
We emulate these concepts by introducing the notion of an $\ep$--efficient 
paths where the points along this path satisfy the reverse triangle inequality 
up to a small multiplicative error. 

\begin{definition} \label{Def:Efficient}
Let $\calZ$ be a metric space, $\gamma \from [a,b] \to \calZ$ be 
quasi-Lipschitz and $R>0$ be a scale so that $|b-a| \lmul R$. 
An $r$--partition of $[a,b]$ is a set of times 
$a=t_0 < t_1 < \ldots < t_m=b$ so that $(t_{i+1}-t_i)\le r$. 
Let $z_i = \gamma(t_i)$. We define the \emph{coarse length of 
$\gamma$ on the scale $r$} to be
\begin{equation} \label{Eq:Efficient}
\Delta(\gamma, r)=
   \min_{\text{$r$--partitions}} \sum_{i=0}^{m-1} d_\calZ (z_i,z_{i+1}).
\end{equation}
We say $\gamma$ is $\ep$--efficient on the scale $R$ if
\begin{equation}
\label{Eq:weak}
\Delta(\gamma,\ep R) \leq d_\calZ\big(\gamma(a),\gamma(b)\big) + O(\ep R).
\end{equation}
\end{definition}

We establish some elementary properties of efficient paths.  
\begin{lemma} \label{Lem:Easy}
Consider a map $\gamma \from [a,b] \to \calZ$.
\begin{enumerate}
\item Suppose $\gamma$ is $\ep$--efficient on scale $R$ and $k$ is a uniformly bounded
integer.  Then for $k$ points $a\leq s_i\leq b$
\[
\sum_{i=0}^{k-1} d_\calZ\big(\gamma(s_i),\gamma(s_{i+1})\big)
=  d_\calZ\big(\gamma(a),\gamma(b)\big) + O(\ep R). 
\]
\item For $[c,d] \subset [a,b]$, if $\gamma$ is $\ep$--efficient at scale $R$
so is $\gamma'=\gamma|_{[c,d]}$.
\item Assume $\calZ = \calZ_1 \times \ldots \times \calZ_l$ equipped with the 
$L^1$--metric and let $\gamma_i$ be the projection of $\gamma$ to $\calZ_i$. 
Then, if $\gamma$ is $\ep$--efficient at scale $R$ so is every $\gamma_i$. 
\item If $\calZ'$ is a subset of $\calZ$ and $\gamma$ is an $\ep^2$--efficient
path that is contained in an $O(\ep^2 R)$--neighborhood of $\calZ'$ then
the closest point projection of $\gamma$ to $\calZ'$ is an $\ep$--efficient path.
\end{enumerate}
\end{lemma}

\begin{proof}
Let $t_0, \ldots, t_m$ be an $\ep R$--partition of $[a,b]$
achieving the minimum in the definition of $\Delta(\gamma,\ep R)$. 
Add the points $s_j$
to the partition.  This 
 will increase the sum in \eqnref{Efficient} by at most 
$O(\ep R)$. In fact, if $t_i \leq s_j \leq t_{i+1}$ then
$|t_i - t_{i+1}| \leq\ep R$ and, since $\gamma$ is quasi-Lipschitz,
\[
d_\calZ\big(\gamma(t_i), \gamma(s_j)\big) 
  + d_\calZ\big(\gamma(s_j), \gamma(t_{i+1})\big) =O(\ep R). 
\]
Since the number of points $s_j$ is uniformly bounded, adding all times
$s_j$ to the partition will increase the sum by at most $O(\ep R)$. 
Now, removing all $t_i$, will only decrease the sum and hence
 part one of the lemma holds. 

To see the second part, as above, let $t_0, \ldots, t_m$ be a set of 
times where the sum in \eqnref{Efficient} is nearly minimal and so that
the times $c$ and $d$ are included in the set $\{t_i\}$. Let $z_i=\gamma(t_i)$.
Letting $c=t_j$ and $d=t_k$ we have 
\begin{align*}
\Delta(\gamma, \ep R) 
&\geq \sum_{i=1}^m d_\calZ(z_i, z_{i+1}) - O(\ep R) \\
  & \geq d_\calZ\big(\gamma(a), \gamma(c)\big) + 
    \sum_{i=j}^{k-1} d_\calZ(z_i, z_{i+1}) +
    d_\calZ\big(\gamma(d), \gamma(b)\big) - O(\ep R) \\
  & \geq d_\calZ\big(\gamma(a), \gamma(c)\big) + \Delta(\gamma', \ep R) +
     d_\calZ\big(\gamma(d), \gamma(b)\big) - O(\ep R).
\end{align*}
Also, by definition, 
\[
\Delta(\gamma, \ep R) 
\leq d_\calZ\big(\gamma(a),\gamma(b)\big) + O(\ep R).
\]
Hence, 
\begin{align*}
\Delta(\gamma', \ep R) &\leq 
 d_\calZ\big(\gamma(a),\gamma(b)\big)
 - d_\calZ\big(\gamma(a),\gamma(c)\big)
 - d_\calZ\big(\gamma(d),\gamma(b)\big) + O(\ep R)\\
 & \leq  d_\calZ\big(\gamma(c),\gamma(d)\big) + O(\ep R).
\end{align*}
This finishes the proof of part 2. 

We prove the third part for $l=2$. The general case is similar. 
Consider the partition $t_0, \ldots, t_m$ that achieves the minimum
for $\Delta(\gamma, \ep R)$. Since $\calZ$ is equipped
with the $L^1$--metric, we have
\begin{equation} \label{Eq:Z1Z2}
\sum_{i=1}^m d_{\calZ_1}(z_i, z_{i+1}) + d_{\calZ_2}(z_i, z_{i+1})
\leq  d_{\calZ_1}\big(\gamma(a), \gamma(b) \big) 
+ d_{\calZ_2}\big(\gamma(a), \gamma(b) \big) + O(\ep R). 
\end{equation}
But, by triangle inequality, we have
\begin{equation} \label{Eq:Z2}
\sum_{i=1}^m d_{\calZ_2}(z_i, z_{i+1})
\geq  d_{\calZ_2}\big(\gamma(a), \gamma(b) \big). 
\end{equation}
Subtracting \eqnref{Z2} from \eqnref{Z1Z2} we obtain
\begin{equation*}
\Delta(\gamma_1, \ep R) \leq \sum_{i=1}^m d_{\calZ_1}(z_i, z_{i+1}) 
\leq d_{\calZ_1}\big(\gamma(a), \gamma(b)\big)  + O(\ep R). 
\end{equation*}

To see the last part, again let $t_1, \ldots, t_m$ be the optimal 
subdivision (note that, in this case, $\gamma$ is  $\ep^2$--efficient). 
Choose a sub-partition $s_1, \ldots, s_l$ so that
\[
|s_{i+1} - s_i| \emul \ep R,
\]
Then, $l \lmul \frac 1 \ep$. Also, let $\gamma'$ be the path
obtain from composing $\gamma$ with the closest point projection to 
$\calZ'$. Let $z_i = \gamma(s_i)$ and $z_i '= \gamma' (s_i)$. 
Then
\begin{align*}
\Delta(\gamma', \ep R) 
& \leq \sum_{i=1}^l d_{\calZ'}(z_i', z_{i+1}' ) \\
& \leq \sum_{i=1}^l d_{\calZ}(z_i', z_i) +
   d_{\calZ}(z_i, z_{i+1} )  + d_{\calZ}(z_{i+1}, z_{i+1}' ) \\
& \leq\Delta(\gamma, \ep R) +  l \cdot O(\ep^2 R) \\
&  \leq \Delta(\gamma, \ep^2 R) +  O(\ep R) 
\leq d_{\calZ'} (z_1', z_l') + O(\ep R).
\end{align*}
In the last inequality, we used the fact that the pairs $z_1, z_1'$ and
$z_l, z_l'$ are $\ep^2 R$--close. This finishes the proof. 
\end{proof}

\begin{definition} \label{Def:size}
A \emph{box} in $\R^n$ is a product of intervals, namely
$B= \prod_{i=1}^n I_i$, where $I_i$ is an interval in $\R$. We say
a box $B$ is \emph{of size $R$} if for every $i$,  
$|I_i| \gmul R$ and if the diameter of $B$ is less than $R$. 
Note that if $B$ is of size $R$ and of size $R'$, then $R \emul R'$. 

A map $f \from B \to \calZ$ from a box of size $R$ in $\R^n$ to a metric space 
$\calZ$ is called \emph{$\ep$--efficient} if, for any geodesic $\gamma \from [a,b] \to B$,  
the path $f\circ \gamma$ is $\ep$--efficient at scales $R$.
\end{definition}

Let $B$ be a box of size $L$ in $\R^n$ and let $\underline B$ be
a central sub box of $B$ with comparable diameter (say a half). 
For any constant $0<R<L$, let $\calB_R$ be a subdivision $\underline B$ to 
boxes of size $R$. That is, 
\begin{enumerate}
\item boxes in $\calB_R$ are of size $R$, 
\item they are contained in $\underline B$ and hence their distance
to the boundary of $B$ is comparable to $L$, 
\item they have disjoint interiors and 
\item their union is $\underline B$. 
\end{enumerate}

For any metric space $\calZ$, we prove that any quasi-Lipschitz maps
from $B$ to $\calZ$ is coarsely differentiable almost everywhere 
in a central box of comparable size: 
 
\begin{theorem}[Coarse Differentiation]  \label{Thm:Differentiable}
For every $\ep_0$, $\theta_0$ and $R_0$ there is $L_0$ so that the following 
holds. For $L \geq L_0$, let $f \from B \to \calZ$ be a quasi-Lipschitz map 
where $B$ is a box of size $L$ in $\R^n$. Then there is a scale $R \geq R_0$
so that the proportion of boxes $B' \in \calB_R$ where $f|_{B'}$ is
$\ep_0$--efficient is at least $(1-\theta_0)$. 
\end{theorem} 

\begin{remark}
Note that the size of the error, $\ep_0 R$, depends on the size of the boxes.
An $\ep_0$--efficient map from a much larger box is allowed to have a much
larger error. What we control is the size of the error as a proportion of the
size of the box. 
\end{remark}

\thmref{Differentiable} is stronger than what we need as we will need only
one efficient box. However, this more general statement may be useful
for other applications of coarse differentiation. 
  
\subsection{Choosing the correct scales}
We first prove a much coarser differentiation statement. 
In a sense, the statement of \thmref{Differentiable} is a direct analogue
of Rademacher's theorem, but the proof of Rademacher's theorem
a direct analogue of proof of \thmref{Coarse} below.

\begin{definition}
A family $\calF$ of geodesics in $\R^n$ is called \emph{locally finite} if, for any 
compact subset $B$ of $\R^n$, only finitely many geodesics in $\calF$ intersect $B$. 
\end{definition}

Let $B$ be a box of size $L$.  Define $\calF_B$ to be the collection of 
restrictions of paths in $\calF$ to $B$ that are long. More precisely, let
\[
\calF_B= \Big\{ \gamma
\ST \gamma = \gamma' \cap B,
\quad \gamma' \in \calF, \quad
|\gamma| \gmul L \Big\}.
\]

For $\gamma \in \calF_B$ we say a set of points $\Gr(\gamma, r)$ along $\gamma$
is an \emph{$r$--grid for $\gamma$} if they subdivide $\gamma$ to 
segments of size exactly $r$ except perhaps for the two end segments which may
have a size less than $r$. An $r$--grid $\Gr(r)$ is a collection of $r$--grids for 
every segment in $\calF_B$. When an $r$--grid 
$\Gr(\gamma,r)= \{p_1, \ldots, p_k \}$ is fixed, we define 
\[
\bDelta(\gamma, r) = \sum_{i=1}^{k-1} d_\calZ(f(p_i), f(p_{i+1})). 
\]
This is essentially the same as the definition of $\Delta$ except the sum is
over a given $r$--grid instead of minimum over all $r$--partitions. 
Given a scale $R$, a segment $\gamma \in \calF_B$ with 
an $\ep R$--grid $\Gr(\ep R)$, we define $\calF(\gamma, R)$ to be the set of all 
subsegments of $\gamma$ of length $R$ that start and end at points in 
$\Gr(\gamma, \ep R)$. We also define 
\[
\calF_B(R) = \bigcup_{\gamma \in \calF_B} \calF(\gamma, R). 
\] 

\begin{theorem}\label{Thm:Coarse}
Let $\calF$ be a locally finite family of geodesics in $\R^n$.  For any 
$\ep>0$, $\theta>0$ and $R_0$, there exist a constant $L_0$ such that the 
following holds. Let $L>L_0$, $B \subset \R^n$ be a box of size $L$ and
$f \from B \to \calZ$ be a quasi-Lipschitz map. 
Then, there exist a scale $R \geq R_0$ and an $\ep R$--grid $\Gr(\ep R)$
such that, for at least $(1-\theta)$ fraction of segments $\gamma' \in \calF_B(R)$, 
\begin{equation} \label{Eq:bDelta}
\bDelta(\gamma', \ep R) \lmul d_\calZ (f(a),f(b))+\ep R. 
\end{equation}
where $a,b$ are the endpoints of $\gamma'$.

\end{theorem}

Note that, \eqnref{bDelta} implies that $\gamma'$ is $\ep$--efficient
at scale $R$. But this statement is more suitable for our proof. 
Morally, the lemma states that $f|_{B}$ is 
nearly affine on scales $R$, up to an error of $O(\ep R)$. 

\begin{remark} In the lemma $R$ depends on $\epsilon$, $\theta$, 
$K$, $C$, and also on $B$. However the proof will find $R$  as one of finitely 
many values  as long as $\epsilon$, $\theta$, $K$ and $C$ are fixed.
\end{remark}

\begin{proof}[Proof of \thmref{Coarse}]
Pick $ r_0 \ge \max\{R_0,C\}$ ($C$ is the additive error in the the definition of a 
quasi-Lipshitz map) and inductively let 
\[
r_m = \frac{r_{m-1}}{\ep}.
\]
Let $M$ be a large positive integer (to be determined below) and let $L_1 = r_M$.  
Choose an arbitrary $r_0$--grid $\Gr(r_0)$ for $\calF_B$
and let $\theta_1$ be the fraction of segments in $\calF_B(r_1)$
for which \eqnref{bDelta} does not hold. If $\theta_1 \leq \theta$ then
we are done.  Thus assume $\theta_1>\theta$. 

For $\gamma \in \calF_B$, we choose an $r_1$--grid 
$\Gr(\gamma, r_1) \subset \Gr(\gamma, r_0)$ as follow: 
Note that $\Gr(\gamma, r_1)$ is essentially a \emph{decomposition} of $\gamma$
into segments of length $r_1$ (except for the subsegments in the ends). 
That is, we are choosing a non-overlapping subset of $\calF(\gamma, r_1)$ so 
that the next segment starts where the previous segment ended. We choose 
the decomposition $\calD(\gamma, r_1)$ so that the proportion $\theta_1(\gamma)$ 
of segments that do not satisfy \eqnref{bDelta} is maximum. Hence, 
the average of these proportions is larger than $\theta_1$.

For $m=1, \ldots, M$, we proceed the same way. If $\theta_m \leq \theta$
we are done. Otherwise, for every $\gamma$, we choose the decomposition 
$\calD(\gamma, r_m)$ where the proportion $\theta_m(\gamma)$ of segments that do not satisfy \eqnref{bDelta}
is maximum and use it to define the $r_m$--grid
$\Gr(\gamma, r_m)$. Again we have
\[
\frac{\sum_\gamma \theta_m(\gamma)}{|\calF_B|} \geq \theta_m > \theta. 
\]

We show that, if $M$ is large enough this contradict the assumption that $f$ 
is quasi-Lipschitz. First, note that:
$$
\bDelta(\gamma,r_{m-1}) - \bDelta(\gamma,r_m)
   \gmul (\ep \, r_m) \, \theta_m(\gamma) |\calD(\gamma, r_m) | \emul
   \ep \,  L \, \theta_m(\gamma).  
$$
After iterating this over $m$ as $m$ goes from $M$ down to $1$ we get
$$
\bDelta(\gamma,r_0)- \bDelta(\gamma,r_M) \gmul \ep \, L  
\sum_{m=1}^M \theta_m(\gamma).  
$$
Using the fact that $f$ is quasi-Lipschitz and $r_0 > C$ we have 
$$
\Delta(\gamma, r_0) \leq \frac L{r_0}(K r_0 + C) \lmul K L.
$$
Hence,
$$
K  L \gmul \ep \, L \sum_{m=1}^M  \theta_m(\gamma),
$$ 
and thus 
$$
\frac{K}\ep \gmul \sum_{m=1}^M\theta_m(\gamma).
$$
Average over all geodesics $\gamma \in \calF_B$ to get
$$
\frac K\ep \gmul 
\sum_{m=1}^M \left (\frac{1}{|\calF_B|}
  \sum_{\gamma\in \calF_B}\theta_m(\gamma)\right) 
  \gmul \sum_{m=1}^M \theta_m \geq M \theta.   
$$
Choosing $M$ large enough we obtain a contradiction. Hence, 
for some $m$, $\theta_m\leq \theta$ and we are done. 
\end{proof}

\begin{proof}[Proof of \thmref{Differentiable}]
Let $L_0$ and $\ep_0<1$ be given. Choose a family $\calF$ of geodesics
in $\R^n$ as follows: pick a finite set of vectors $\calV$ in the unit sphere 
$S^{n-1}\subset \R^n$ that is $(\ep_0)^2$ dense in $S^{n-1}$
with the size $|\calV| \emul \frac 1{(\ep_0)^2}$.
For a direction $\vec v \in \calV$, let $\calF_{\vec v}$ be a family 
of parallel lines in the direction $\vec v$ 
where the distance between nearby lines is comparable to $1$. Then
$$
\calF= \bigcup_{\vec v \in \calV} \calF_{\vec v}
$$
is a locally finite family of geodesic in $\R^n$. Let 
$$
\ep=\ep_0^2, \quad\text{and}\quad \theta \ll  \theta_0 \, \ep^{n+2}.
$$ 
Apply \thmref{Coarse} to obtain the constants $L_0$. 
Assume a box $B$ of size $L\ge L_0$ and a quasi-Lipschitz map 
$f \from B \to \calZ$ are given and let $R$ be the scale obtained from 
\thmref{Coarse}. 

Let $\calB_R$ be a collection of disjoint sub-boxes of $B$ giving
a decomposition of a central box in $B$ as in the statement of \thmref{Differentiable}.
Let $B' \in \calB$ be a box that contains a geodesic $\beta$ that is not 
$\ep_0$--efficient at scale $R$. Let $\vec v$ be the direction closest to the 
direction of $\beta$. Let 
\[
\NE(B', \vec v) \subset\calF_B(R)
\] 
be the set of geodesic segments in $\calF_{B,R}$ that are in the direction 
of $\vec v$, intersect $B'$ and are not $\ep_0$--efficient on scale $R$.

\subsection*{Claim:} Every geodesic in $\calF_{\vec v}$ that intersects 
an $\ep R$--neighborhood of $\beta$ contains a segment in 
$\NE(B', \vec v)$.

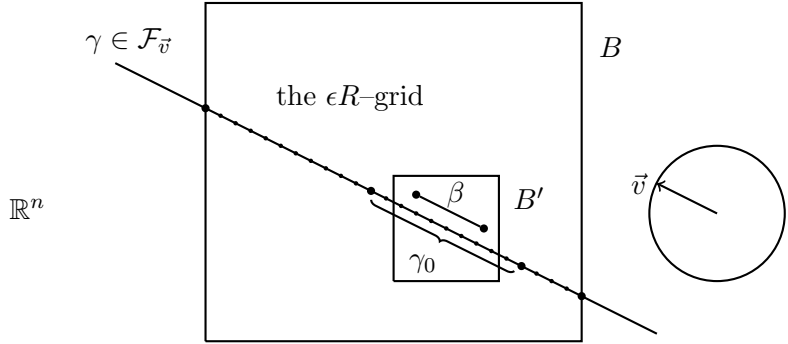
\begin{figure}[ht]
\setlength{\unitlength}{0.01\linewidth}
\begin{picture}(100, 37) 

\put(20,0){
\begin{tikzpicture} 
[thick, 
    scale=0.10,
    vertex/.style={circle,draw,fill=black,thick, inner sep=0pt,minimum size= .8 mm},
    vtx/.style={circle,draw,fill=red, inner sep=0pt,minimum size= .3 mm}]
    
   \node[vertex] (bl) at (58,21.5)   {}; 
   \node[vertex] (br) at (67,17)   {};

   \node[vtx] (g1) at (32,32)   {};
   \node[vtx] (g2) at (34,31)   {};
   \node[vtx] (g3) at (36,30)   {}; 
   \node[vtx] (g4) at (38,29)   {}; 
   \node[vtx] (g5) at (40,28)   {}; 
   \node[vtx] (g6) at (42,27)   {}; 
   \node[vtx] (g7) at (44,26)   {}; 
   \node[vtx] (g8) at (46,25)   {}; 
   \node[vtx] (g9) at (48,24)   {};   
   \node[vtx] (h1) at (50,23)   {};
   \node[vtx] (h2) at (54,21)   {};
   \node[vtx] (h3) at (56,20)   {}; 
   \node[vtx] (h4) at (58,19)   {}; 
   \node[vtx] (h5) at (60,18)   {}; 
   \node[vtx] (h6) at (62,17)   {}; 
   \node[vtx] (h7) at (64,16)   {}; 
   \node[vtx] (h8) at (66,15)   {}; 
   \node[vtx] (h9) at (68,14)   {};   
   \node[vtx] (h1) at (70,13)   {};
   \node[vtx] (h2) at (74,11)   {};
   \node[vtx] (h3) at (76,10)   {}; 
   \node[vtx] (h4) at (78,9)   {}; 

   \node[vertex] (c1) at (30,33)   {};
   \node[vertex] (c2) at (52,22)   {};   
   \node[vertex] (c3) at (72,12)   {};
   \node[vertex] (c4) at (80,8)   {};

   \node(v0) at (98,14)   {};
   \node(v1) at (88,19)   {};   

   \draw(30, 2) -- (80, 2) -- (80, 47) -- (30, 47) -- (30, 2);
   \draw(55, 10) -- (69, 10) -- (69, 24) -- (55, 24) -- (55, 10);
   \draw(18, 39) -- (90, 3);
   \draw(bl) -- (br);

\draw[decorate, decoration= {brace,mirror,raise=3pt}] (c2) -- (c3); 

\draw [->]  (98,19) -- (90,23)  node [left] {$\vec v$};
\draw (98,19)  circle (9);

\end{tikzpicture}}
 
\put(10,13){$\R^n$} 
\put(18,31){$\gamma \in \calF_{\vec v}$} 
\put(72,30){$B$} 
\put(63,14){$B'$} 
\put(56,14.8){$\beta$} 
\put(52,7.9){$\gamma_0$} 
\put(38,25){the $\ep R$--grid} 

\end{picture}
\caption{The arc $\gamma$ is in a $\ep R$--neighborhood of $\beta_0$.}
\label{Fig:gamma} 
\end{figure}

\emph{Proof of claim:} 
Assume $\gamma \in \calF_{\vec v}$ intersects an $\ep R$--neighborhood of 
$\beta$. The condition (2) of description of $\calB_R$ implies that 
$|\gamma \cap B| \gmul L$ and its subsegments of length $R$  
that start and end in $\Gr(\ep R)$ are included 
$\calF_{B,R}$. Since, the difference between the direction of $\gamma$ and
$\beta$ is at most $\ep=\ep_0^2$, in fact $\beta$ is contained
in an $\ep R$--neighborhood of $\gamma$. Also, the length of
$\beta$ is less than the diameter of $B'$ which is less than $R$. 
Hence, there is a segment $\gamma_0 \in \calF_B(\gamma, R)$ of
length $R$ where $\beta$ is included in an $\ep R$--neighborhood of $\gamma_0$
(refer to \figref{gamma}).
We show that if $\beta$ is not $\ep_0$--efficient on scale $R$, then $\gamma_0$ 
will not be $\ep$--efficient on scale $R$  which is what we claimed. 

Assume, for contradiction, that $\gamma_0$ is $\ep$--efficient on scale $R$.
Then every sub-segment of $\gamma_0$ is also efficient on scale $R$ 
(\lemref{Easy}). Choose a subsegment $\gamma_1$ of $\gamma_0$ 
so that the end points of $\gamma_1$ and $\beta$ are $\ep R$--close. 
We now apply the last conclusion of \lemref{Easy} with $\calZ'=f(\gamma_0)$ to 
conclude that $\beta$ is $\ep_0$--efficient, a contradiction. \qed

Let $\calF_{B'}(R)$ be the subset of $\calF_B(R)$ consisting of
segments that intersect $B'$. In every direction $\vec v \in \calV$
there are at most $\frac{R^{n-1}} \ep$ segments in $\calF_{B'}(R)$.
This is because a cross section of $B'$ perpendicular to $\vec v$
has an area at most $R^{n-1}$ and the grid has size $\ep R$.
Assuming $B'$ contains an non-efficient segment $\beta$, the number 
of geodesics in $\calF_{\vec v}$ that intersect an $\ep R$--neighborhood
of $\beta$ is of order of $(\ep R)^{n-1}$ (which is the area of a cross section
of an $\ep R$--neighborhood of $\beta$ perpendicular to $\vec v$). 
That is
\[
\frac{\big|\NE(B', \vec v)\big|} { |\calF_{B'}(R)|} 
\gmul \frac 1{|\calV|} \cdot \frac {(\ep R)^{n-1}}{\frac{R^{n-1}}{\ep}} \emul \ep^{n+2}.
\]
Note that a definite proportion of segments in $\calF_B(R)$ intersect
some box $B' \in \calB$ and each segment in $\calF_B(R)$ intersects
at most a uniform umber of boxes. Hence
\[
 |\calF_B(R)| \emul \sum_{B' \in \calB}  |\calF_{B'}(R)|.
\]
Define
\[
\NE(B') = \bigcup_{\vec v \in \calV} \NE(B', \vec v)
\qquad\text{and}\quad
\NE = \bigcup_{B' \in \calB} \NE(B')
\]
Assume the proportion of boxes $B'$ that contain a non-efficient segment is
larger than $\theta_0$. Since the sizes of $\calF_{B'}(R)$ are comparable for 
every $B'$, we have
\begin{align*}
\frac{|\NE|}{|\calF_B(R)|} 
\geq \frac 1{|\calB|} \sum_{B' \in \calB} \frac{\big|\NE(B')\big|} { |\calF_{B'}(R)|}
\gmul \frac{\theta_0 |\calB|}{|\calB|} \ep^{n+2} \gg \theta.
\end{align*}
The contradiction finishes the proof. 
\end{proof}

\subsection{Efficient map into a hyperbolic space} \label{Sec:Eff-Hyp}
The following is the first use of efficient paths when the target is Gromov 
hyperbolic and is similar to the familiar Morse argument.

\begin{lemma}
\label{Lem:Close-In-Hyp}
Suppose $\X$ is a Gromov hyperbolic space and $\gamma \from [a,b]\to \X$ is 
$\epsilon$--efficient on scale $R$.  Then $\gamma$ stays in an 
$O(\epsilon R)$--neighborhood of a geodesic $\ell$ joining $\gamma(a)$ 
and $\gamma(b)$.   
\end{lemma}

\begin{proof}
Suppose, for a large $M$, that the path $\gamma$ leaves an 
$M \epsilon R$--neighborhood of $\ell$. Fix a constant $D_0 \geq C$.
We can find times  $c,d \in [a,b]$  so that at times $c,d$, $\gamma$ is distance $D_0$ from $\ell$;    for $t \in [c,d]$, 
$\gamma(t)$ is at least $D_0$ away from $\ell$, and so that in between $c$ 
and $d$ the path $\gamma$ travels to a point at of a distance 
$M\epsilon R$ from $\ell$.   

By \lemref{Easy}, $\gamma'= \gamma|_{[c,d]}$ is still $\ep$--efficient. 
Let $c=t_1 < \ldots < t_N=d$ be a partition so that, for $z_i = \gamma(t_i)$,
\begin{equation} \label{Eq:Best}
\Delta(\gamma', \ep R) \gmul
\sum_i d_\X (z_{i+1}, z_i ).
\end{equation}
(In fact, by definition of $\Delta$, we can choose $t_i$ so that the two sides 
are equal. However, we are about to modify the partition $t_i$.) 
We can remove some of the times $t_i$ so that
\[
d_\X (z_{i+1}, z_i) \emul \ep R. 
\]
Note that after removing points from the partition, \eqnref{Best} still holds.
Since $t_{i+1}-t_i\emul \epsilon R$ we have $N \gmul M$.  
By the contraction property of hyperbolic spaces, there is a uniform constant
$D_1$ and a projection map $\pi\from \X\to \ell$ such that 
\[
d_\X\big(\pi(z_i),\pi(z_{i+1})\big)\leq D_1.
\] 
By using these projected points to $\ell$ and since $\gamma(c)$ and $\gamma(d)$ are at distance $N_0$ from $\ell$   we have 
\[
d_\X \big(\gamma(c),\gamma(d) \big) \leq  2D_0 + N D_1.
\] 
On the other hand   
\[
\Delta(\gamma',\epsilon R)\gmul \sum d_\X (z_{i+1},z_i) \gmul N\epsilon R .
\]  
From the assumption that $\gamma'$ is $\epsilon$--efficient on scale $R$,  
we have
\[
2 D_0+N D_1 +O(\ep R)  \gmul N \ep R. 
\]  
This implies $N$ is uniformly bounded. Hence, $M$ is also
uniformly bounded. This finishes the proof.
\end{proof} 

We now consider an efficient map from a box to a hyperbolic space. 
First we need the following lemma.

\begin{lemma}
\label{Lem:Convex}
Given $n$ and $N$, there is $\sigma=\sigma(n,N)>0$ such that for each $L$, if  
$\{C_i\}$ is a collection of $N$ convex bodies in $\R^n$  that cover a ball $B$ 
of radius $L$, then some $C_i$ contains a ball of radius $\sigma L$. 
\end{lemma}
\begin{proof}
For a convex set $C$, let $R=R(C)$ be the out-radius of a convex set: the radius of 
the smallest ball that contains it. Let $r=r(C)$ be the in-radius: the largest ball 
contained in the set and $w=w(C)$ be the width: the minimum  distance between 
supporting hyperplanes.

From Theorem 1 in \cite{cifre:OC} we have, for some $\kappa=\kappa(n)>0$,  
and any convex set $C$ that 
\[
\text{Vol}(C)\leq \kappa R^n\int_0^{\arcsin \frac{w}{2R}} \cos^n\theta d\theta
  \leq \kappa R^n\arcsin \frac{w}{2R}\leq \kappa R^n\arcsin \frac{c\,r}{2R}.
\] 
The last inequality follows from the Steinhagen inequality, which states that 
there is a constant $c=c(n)>0$,  such that 
\[
w\leq c\, r.
\]

Since the convex sets $C_i$ cover the ball of radius $L$,  for some 
$c'>0$, there is some $C=C_i$ with
\[
\Vol(C)\geq  c'L^n/N.
\] 
This implies that $R \geq c''L$, for some constant $c''=c''(N)>0$.

We will show $r\geq \sigma L $ by arguing in two cases. Assume, 
\[
\arcsin \frac{c\, r}{2R}\geq\frac{\pi}{4}.
\] 
Then $\frac{c\, r}{2R}\geq \frac{\sqrt{2}}{2}$ and so 
\[
c\,r\geq  \sqrt{2}R\geq \sqrt{2}c''L,
\]
and we are done by taking $\sigma=\frac{\sqrt{2}c''}{c}$.  
Now assume 
\[
\arcsin \frac{c\,r}{2R}\leq\frac{\pi}{4}
\] 
so that 
\[
\arcsin \frac{c\,r}{2R}\leq \frac{c\,r}{R}.
\] 
But then 
\[
c'L^n/N\leq \Vol(C)\leq \kappa\, R^n\frac{c\,r}{R}= \kappa \, c \, R^{n-1}\,r
\] 
and so 
\[
r\geq \frac{c'}{c\, N\kappa} \left(\frac LR\right)^{n-1}  L,
\] 
and again we are done by taking $\sigma=\frac{c' }{c\,(c'')^{n-1}N\kappa }$.
\end{proof}

\begin{proposition} \label{Prop:Efficient-to-Hyperbolic}
Suppose $\calZ$ is a Gromov hyperbolic space and $f \from B \to \calZ$
is an $\ep$--efficient map from a box of size $R$ in $\R^n$ to $\calZ$. 
Then, there is a sub-box $B' \subset B$ with $|B'| \emul |B|$, so that 
the image $f(B')$ lies in an $O(\ep R)$--neighborhood of a line $\ell'$ in $\calZ$. 
\end{proposition}

\begin{proof}
By taking a sub-box we assume $B=[0,R]^n$ and let $\ell_i$ be the edges
of the box $B$. Given a line $\ell \subset B$ denote by $\ell'$ a geodesic 
in $\calZ$ joining the $f$ image of its endpoints. 

We first prove by induction on $n$ that, for any $q\in B$, $f(q)$ is 
within $O(\ep R)$ of a point in some $\ell'_i$.  We start with $n=2$ and  
$\ell_1,\ell_2,\ell_3,\ell_4$  the four edges of $B$ arranged in counterclockwise 
order.   We have by \lemref{Close-In-Hyp} that each point of $f(\ell_i)$ is 
within $O(\ep R)$ of $\ell'_i$. Now let $\ell_q$ be the line  through $q$ parallel 
to  $\ell_1$.  Take the  rectangle with sides $\ell_q,\ell_1$ and subsegments 
$m_2\subset \ell_2$ and $m_4\subset \ell_4$.  \lemref{Close-In-Hyp} implies
that the end points of $m_2'$ are within $O(\ep R)$ of $\ell_2'$ and hence
$m_2'$ are within $O(\ep R)$ of $\ell_2'$. The same holds for $m_4'$ and 
$\ell_4'$.  The quadrilateral bounded by  $\ell_1', m_2',\ell_q', m_4'$ is $O(1)$ 
thin, which implies that $f(q)$ is within $O(\epsilon R)$ of  one of the other 
three sides and therefore within $O(\epsilon R)$ of one of the $\ell_i'$.  

Now suppose the statement is  true for boxes in $\R^{n-1}$ and  
$B\subset \R^n$.   Take again the geodesics $\ell_i$  that correspond to the 
edges of the box $B$ and any point $q\in B$.  It lies on a face $B_q^{n-1}$ 
parallel to the a face of $B$. Let $\tau_i$ be the edges of $B_q^{n-1}$.  
By induction, $f(q)$ lies within $O(\epsilon R)$ of some $\tau_i'$.  Since each 
$\tau_i$ itself lies in an $n-1$ dimensional face, again by induction, each point 
of  $f(\tau_i)$ lies within $O(\epsilon R)$ of  the union of $\ell_i'$.  Thus $f(q)$ is 
within $O(\epsilon R)$ of some $ \ell_i'$, completing the induction step.   

Now, fix any $n$ and one of the  geodesics $ \ell_i'$.  If $n+1$ points 
$q_1,\ldots q_{n+1}$ span an $n$ simplex $\Lambda$ and  are such that 
each $f(q_j)$ is within $O(\epsilon R)$ of $\ell_i'$, then the image under $f$ of every 
point of $\Lambda$ is within $O(\epsilon R)$ of $\ell_i'$.    By the 
Caratheodory theorem, the convex hull of the set of points mapped within $O(\epsilon R)$ 
of $ \ell_i'$   is the union of such simplices, and therefore the convex hull is a 
convex set of points  mapped within $O(\epsilon R)$ of $\ell_i'$. We conclude by 
\lemref{Convex} that there is a box $B'$ with $|B'|\emul |B|$ consisting of 
points mapped within $O(\epsilon R)$ of one of the $\ell_i'$. 
\end{proof}

\section{Combinatorial Model} \label{Sec:Model}
Let $S$ be a possibly disconnected surface of finite hyperbolic type. 
Define the complexity of $S$ to be
\begin{equation} \label{Eq:Complexity}
\xi(S) = \sum_W (3{\sf g}_W + {\sf p}_W-4),
\end{equation}
where the sum is over all connected components $W$ of $S$,
${\sf g}_W$ is the genus of $W$ and ${\sf p}_W$ is the number of punctures. 

Let $(\T(S), d_\T)$ represent the \Teich space equipped with
the \Teich metric, $(\T(S), \d_{W\!P})$ represent \Teich space equipped with 
the Weil-Petersson metric and $(\Mod(S), d_W)$ represent the mapping class 
group equipped with the word metric. We construct combinatorial
models for these spaces. 

Let $\calP(S)$ be the space of pants decompositions of $S$. That is,
a point $P$ in $\calP(S)$ is a free homotopy class of maximum number of
disjoint essential simple closed curves. Define a marking 
$(P, \{\tau_\alpha\}_{\alpha \in P})$ to be a pants decomposition together
with a transverse curve $\tau_\alpha$, for each pants curve $\alpha$. 
The transverse curves are assumed to be disjoint from other curves in 
$P$ and to intersect $\alpha$ minimally (see \cite{minsky:CCII} for more details).
The space of all markings is denoted by $\calM(S)$. An augmented marking
$(P, \{\tau_\alpha\}_{\alpha \in P}, \{\ell_\alpha\}_{\alpha \in P})$ 
is a marking together with a positive real number $l_\alpha$ (length of
$\alpha$) associated to to every pants curve $\alpha$. The length 
of each curve is assumed to be less than the Bers constant for the surface $S$.
The space of augmented markings is denoted by $\AM(S)$
(see \cite{rafi:HT} and also \cite{durham:AM} for slightly different definition
and extensive discussion of $\AM(S)$). 
We will use these spaces as combinatorial models for, respectively,
$(\T(S), d_{W\!P})$, $(\Mod(S), d_W)$ and $(\T(S), d_\T)$. 
Assume $\X= \X(S)$ is one of these model spaces. Later in this section we will 
equip $\X$ with a coarse metric. 

\subsection{Curve complex} \label{Sec:Curve-Cpmplex}
Let $W$ be a subsurface of $S$. We always assume a subsurface is connected
(unless specified otherwise) and that the embedding $W\subset S$ induces an 
injective map $\pi_1(W) \to \pi_1(S)$. We also exclude the cases
where $W$ is a thrice-punctured sphere or an annulus going around a puncture. 

Let $\C(W)$ be the curve graph of $W$ with metric $d_{\calC(W)}$.
This is a graph where the vertices are free homotopy class of 
nom-trivial non peripheral simple closed curve (henceforth, simply
referred to as curves) and edges are pairs of curves intersecting
minimally (see \cite{minsky:CCII} for precise definition and discussion). 
We make a special definition for the case of annuli. For an annulus $A$, 
\begin{itemize}
\item $\C(A)$ is a horoball in $\HH^2$ when $\X=\AM(S)$. 
\item $\C(A)$ is $\ZZ$ when $X=\calM(S)$. 
\item $\C(A)$ is a point when $\X=\calP(S)$. 
\end{itemize}
The curve complex of every subsurface is Gromov hyperbolic in all cases. 
This is clear when $W$ is an annulus and is a theorem of 
Masur-Minsky \cite{minsky:CCI} in other cases. 

For every subsurface $W$ of $S$, there is a coarsely defined projection map
(see \cite{minsky:CCII} for general discussion and \cite{durham:AM} for the case
of augmented markings)
\[
\pi_W \from \X \to \C(W). 
\] 
We sketch the definition here. Assume first that $W$ is not an annulus. 
Given $x\in \X$ (recall that in all three cases $x$ contains a pants decomposition 
which we denote by $P_x$) 
choose any pants curve $\gamma \in P_x$ that intersects $W$.   
If $\gamma\subset W$ then choose the projection to be $\gamma$.  
If $\gamma$ is not contained in $W$ then $\gamma\cap W$ is a collection of 
arcs  with endpoints on $\partial   W$.  Choose one such arc and perform a 
surgery using this arc and a subarc of $\partial W$ to find a point in $\C(W)$. 
The choice of different arcs or different choices of intersecting pants curves 
determines a set of diameter $2$ in $\C(W)$; hence the projection is coarsely 
defined.

For annuli $A$ the definition is slightly different. When $\X= \calP(S)$
the projection map is trivially defined since $C(A)$ is just a point. 
When $\X= \calM(S)$, consider the annular cover $\tilde A$ of $S$ associated 
to $A$. Identify the space of arcs in $\tilde A$ (homotopy classes of arc connecting 
different boundaries of $\tilde A$ relative to their end points) with $\ZZ$ 
by identifying some arc $\omega_0$ with zero and sending every other arc
$\omega$ to the signed intersection number between $\omega$ and
$\omega_0$. Define $\pi_A(x)$ by lifting the pants deposition $P_x$ and transverse 
curves $\tau_\alpha$ to $\tilde A$. At least one of these curves lifts to an 
arc connecting different boundaries of $\tilde A$ and different ones have bounded intersection number. Hence the map is coarsely defined. 
We refer to this number as the twisting number of $x$ around $\alpha$ and
denote it by $\twist_\alpha(x)$, which coarsely defined integer. 

Now consider the case $\X=\calA\calM(S)$. Let $\sf B$ be the Bers constant
of the surface $S$.  For an annulus $A$, we identify $\C(A)$ with the subset 
$H \subset \R^2$ of all points in $\R^2$ where the $y$--coordinate is larger 
than $1/{\sf B}$.  Note that, for an augmented marking 
$x=(P, \{\tau_\alpha\}, \{l_\alpha\})$, the twisting number $\twist_\alpha(x)$ 
can still be defined as above. If the core curve of $A$ is in $P_x$ we define 
$$
\pi_A(x) = \big( \twist_\alpha(x), 1/l_\alpha \big),
$$
otherwise
$$
\pi_A(x) = \big( \twist_\alpha(x), 1/B \big).
$$

Also, for subsurfaces $U$ and $V$ we have a projection map
\[
\pi_{U,V} \from \C(U) \to \C(V). 
\]
which is defined on the subset of $\C(U)$ consisting of curves that
intersect $V$. Here $U$ is non-annular; for an annulus $A$, 
elements of $\C(A)$ cannot be projected to other subsurfaces. 
When the context is clear, we denote all these projection 
maps simply by $\pi$. By construction, all projection maps are
quasi-Lipschitz.

\subsection{Distance Formula} \label{Sec:Distance}
For $x,y\in \X$, define the $W$--projection distance between $x$ and $y$ to be:
\[
d_W(x,y) = d_{\C(W)}\big(\pi_W(x), \pi_W(y) \big).
\]
In fact, when $W$ is an annulus with core curve $\alpha$, we sometimes
denote this distance with $d_\alpha(x,y)$. 
We define the \emph{distance} in $\X$ using these projection distances. 
For a threshold $T>0$ large enough, define
\begin{equation} \label{Eq:Distance}
d_\X(x,y) = \sum_{W \in \calW_T(x,y)} d_W(x,y),
\end{equation}
where $\calW_T(x,y)$ is the set of subsurfaces with $d_W(x,y)\geq T$. 
This is not a real metric since the distance between different points
may be zero and the triangle inequality does not hold. However, it is symmetric 
and the triangle inequality holds up to a multiplicative error. That is, 
for $x,y,z \in \X$
\[
d_\X(x,y) + d_\X(y,z) \gmul d_X(x,z). 
\]
Also, changing the threshold changes the metric by only uniform
additive and multiplicative constants. That is, for $T' \geq T$ we have 
(see \cite{minsky:CCII, rafi:CM})
\begin{equation} \label{Eq:Threshold} 
\sum_{W \in \calW_T(x,y)} d_W(x,y) 
   \emul \sum_{W \in \calW_{T'}(x,y)} d_W(x,y).
\end{equation}
Even though this is not a metric, it makes sense to say $\X$ is quasi-isometric
to another metric space. In fact, in the category of metric spaces up to
quasi-isometry, this notion of distance is completely adequate. 
We fix a threshold $T$ once and for all so that $d_\X(x,y)$ is
a well defined number for all $x,y \in \X$. The threshold $T$
needs to be large enough so that statement in the rest of this subsection hold. 

There is a coarsely defined map
\[
(\T(S), d_{W\!P}) \to  \calP(S)
\]
sending a Riemann surface $X$ to the shortest pant decomposition in $X$ 
which is, by \cite{brock:WV}, a quasi-isometry. Hence $\calP(S)$ with the 
above metric is a combinatorial model for the Weil-Petersson metric. 

A point in $(\Mod(S), d_W)$ can be coarsely represented as a marking 
\cite{minsky:CCI}. That is, there is a coarsely defined map 
\[
(\Mod(S), d_W) \to  \calM(S),
\]
which can be defined by, for example,  fixing a point $x_0 \in \calM(S)$ 
and sending a mapping class $\phi \in \Mod(S)$ to the marking $\phi(x_0)$. 
It is shown in \cite{minsky:CCII} that this map is a quasi-isometry. 

A point in $(\T(S), d_\T)$ can be coarsely represented as an augmented 
marking  \cite{rafi:HT}. That is, there is a coarsely defined map 
\[
(\T(S), d_\T) \to  \calA\calM(S)
\]
defined as follows. A point in $X$ in \Teich space is mapped to the augmented
marking $x=(P, \{\tau_\alpha\}, \{l_\alpha\})$ where $P$ is the shorts
pants decomposition in $X$, for $\alpha \in P$, $\tau_\alpha$ is the
shortest transverse curve to $\alpha$ in $X$ and $l_\alpha$ is the
hyperbolic length of $\alpha$ in $X$. It follows from \cite{rafi:CM}
that this map is a quasi-isometry. (Again, see \cite{durham:AM} for
more details in this case.) 

By $(\X(S), d_\X)$ we denote one of the model spaces above. 
When the context is clear, we use $\X$ instead of $\X(S)$. However,  often 
we need to talk about $\X(W)$ when $W$ is a 
subsurface of $S$. For example, if $\X(S)$ is the space of pants
decompositions of $S$, then $\X(W)$ is the space of pants decompositions
of $W$.

\subsection{Bounded projection, consistency and realization} 
\label{Sec:Consistency} 
In this section, we review some properties of the projection maps. 
We will also drive a coarse characterization of the image of the
curve complex projections of points in $\X$ similar to \cite{minsky:MR}.

We start with a Theorem from \cite{minsky:CCII}. 

\begin{theorem}[Bounded Geodesic Image Theorem]
\label{Thm:BGIT}
There exists a constant $M_0$ so that the following holds. 
Assume $V \subsetneq  U$ are subsurfaces of $S$ and 
$\beta_1,\ldots,\beta_k$ is a geodesic in $\calC(U)$.
Then either there is some $\beta_j$ that is disjoint from $V$ or
$d_V(\beta_1,\beta_k)\leq M_0$.
\end{theorem}

Masur-Minsky proved this theorem for $\M(S)$ but it then holds for augmented 
markings as well (and trivially it holds for $\calP(S)$). The only case to check is when 
$V$ is an annulus with a core curve $\alpha$ and $\C(V)$ is a horoball. Then, 
applying the $\calM(S)$ version of \thmref{BGIT}, we have that either
some $\beta_j$ is disjoint from $\alpha$ or 
\[
\twist_\alpha(\beta_1) \eadd \twist_\alpha(\beta_k).
\]
Also, 
\[
\pi_V(\beta_i) = (\twist_\alpha(\beta_i), 1/B). 
\]
Hence, 
\[
d_\alpha(\beta_1, \beta_k) \eadd 
\log |\twist_\alpha(\beta_1) - \twist_\alpha(\beta_k)|=O(1). 
\]
Thus, \thmref{BGIT} holds for $\AM(S)$ as well. 

The other important property of the projection maps is the consistency
and realization result of Behrstock-Kleiner-Mosher-Minsky. First we recall that 
for subsurfaces $U,V$ the notation  $U\pitchfork V$ means that 
$U\cap V\neq\emptyset$ and 
neither is contained in the other. Consider the following 
\emph{consistency condition} on a tuple 
\[
{\sf z} \in \prod_U \C(U).
\]
Denote the coordinate of $\sf z$ in $\C(U)$ with $z_U$.
For a constant $M$, we say ${\sf z}$ is $M$--consistent  if
\begin{enumerate}
\item Whenever $U \pitchfork V$, 
$$
\min \big(d_U (z_U , \partial V ), d_V (z_V , \partial U )\big) \leq M.
$$
\item  If $V \subsetneq U$ , then 
$$
\min \big(d_U(z_U,\partial V), d_V (z_V , z_U) \big) \leq M.
$$ 
\end{enumerate}

To any $z \in \X$, the \emph{tuple of projections of $z$} is a tuple $\sf z$
so that $z_U = \pi_U(z)$. The following, in case $\X$ is $\calM(S)$ and 
$\calP(S)$,  is Theorem 4.3 in \cite{minsky:MR}. 
However it holds true for $\AM(S)$ as well.

\begin{theorem}[Consistency and Realization]  \label{Thm:Consistency}  
The tuples that are consistent are essentially those that are
tuples of projections. More precisely, there is a constant $M_1$, so that 
\begin{enumerate}
\item For $z \in \X$,  the tuple of projection $\sf z$ of $z$ is
$M_1$--consistent. 
\item If a tuple ${\sf z}$ is $M$--consistent for some uniform $M$,  
then there is a \emph{realization} $z \in \X$ so that
$$
\forall U \quad d_U(z,z_U)=O(1).
$$
\end{enumerate}
\end{theorem}

\begin{proof}
As mentioned before, this is known for $\calP(S)$ and $\calM(S)$. 
We verify the theorem in the case of $\AM(S)$. 

First we check part (1). For any $z \in \X$, let $P_z$ be the associated
pants decomposition. The non-annular projections of $z$ are
the same as projections of $P_z$ and by the $\calP(S)$--case of
\thmref{Consistency} these projections are consistent. 
Let $A$ be an annulus and $U$ be any other 
surface intersecting $A$. As in \thmref{BGIT}, 
the consistency still holds because 
\begin{itemize}
\item the image of $\pi_{U,A}$ is always on the boundary the horocycle
$\C(A)$. 
\item the distance in $\C(A)$ between two points on the boundary is the 
log of the twisting difference, and
\item the twisting distance is bounded as a consequence of the consistency 
theorem for $\M(S)$. 
\end{itemize}
That is, the consistency constant for $\AM(S)$ is no larger than
that of $\calM(S)$. 

To see part (2) we need to construct an augmented marking from
a consistent tuple ${\sf z}$. Use the realization part of  \thmref{Consistency} 
for $\calP(S)$ to construct a pants decomposition $P_0$ so that,
for every subsurface $U$ that is not an annulus, 
$d_{U}(P_0, z_U)=O(1)$. 
Still, for some curves $\alpha$, the projection of $P$ to $\C(\alpha)$
may not be close to $z_\alpha$. We show that the set
of such curves is a multi-curve. 

Let $\alpha$ and $\beta$ be two curves that intersect $P_0$ so that
$d_{\alpha}(P_0, z_\alpha)$ and $d_{\beta}(P_0, z_\beta)$
are both large. We show that $\alpha$ and $\beta$ are disjoint. Assume,
for contradiction, that they intersect and let $U$ be the surface
they fill. Then $z_U$ intersects either $\alpha$ to $\beta$ 
(say $\alpha$ without loss of generality). By construction of $P_0$
\[
d_{U} (P_0, z_U)=O(1) 
\quad \xrightarrow{\text{ $\pi_{U,A}$ is quasi-Lipschitz }}\quad 
d_{\alpha} (P_0, z_U)=O(1),
\]
and by consistency of coordinates of $\sf z$
\[
d_{\alpha}(z_U, z_\alpha)= O(1). 
\]
Now, the triangle inequality implies that 
$d_{\alpha} (P_0, z_\alpha)=O(1)$ which is a contradiction. 

We also note that, for any such curve $\alpha$ (where 
$d_\alpha(P_0, z_\alpha)$ is large) and for every subsurface $U$ intersecting 
$\alpha$, 
\[
d_{U}(\alpha, P_0)\eadd d_{U}(\alpha, z_U)= O(1).
\]
This is because, if $z_U$ is far from $\alpha$ in $\C(U)$, then 
the projection of $z_U$ to $\C(\alpha)$ is defined and
is ($\sf z$ is constant) near $z_\alpha$. But, as above, 
$d_U (P_0, z_U)=O(1)$ which implies $d_{\alpha} (P_0, z_U)=O(1)$. 
This is a contradiction. 

Let $\balpha$ be the multi curve consisting of all the curves 
above. Since, $d_U(P_0, \balpha)=O(1)$ for every non-annular 
subsurface $U$, $\balpha$ can be extended to a 
pants decomposition $P$ with $d_U(P_0, P)=O(1)$ for every non-annular 
subsurface $U$. That is, 
\[
d_V(P, z_V)=O(1), \qquad \text{for every subsurface $V$ intersecting $P$}.
\]

We now complete $P$ into an augmented marking. 
For a curve $\beta \in P$, write $z_\beta \in \C(\beta)$ as
\[
z_\beta = (t_\beta, l_\beta ),
\]
where $t_\beta$ is an integer and $l_\beta$ is a real number less than
the Bers constant. Let $\tau_\beta$ be a curve intersecting
$\beta$ minimally and disjoint from other curves in $P$ with
$\twist_\beta(\tau_\beta) \eadd t_\beta$ (this can always be achieved 
by applying Dehn twists around $\beta$). Now, 
\[
x= (P, \{\tau_\beta\}_{\beta \in P}, \{l_\beta\}_{\beta_P}).
\]
is the desired augmented marking. This is because, for $\beta \in P$,  
the projections to $\C(\beta)$ are close to $z_\beta$ by construction
and, for every other subsurface $V$, the projection of $x$ to $V$ 
is the same as the projection of $P$ to $V$. 
\end{proof}

The following corollary of this theorem will be useful later. 

\begin{corollary} \label{Cor:Far}
Let $U$ and $V$ be two subsurfaces where $\partial V$ intersects
$U$. The for any $x \in X$, if
\[
d_U(x_U, \partial V) > M_1
\qquad\text{then}\qquad
d_V(x_V, x_U)=O(1). 
\]
\end{corollary}

\begin{proof}
There are two cases. If $V \subsetneq U$, then this is immediate from
part two of the consistency condition.

Otherwise, $\partial U$ intersects $V$ and has a defined projection to
$\C(V)$. Since $x_U$ is disjoint from $\partial U$, $d_V(x_U, \partial U)=O(1)$. 
On the other hand the assumption of the corollary and part one of the consistency condition say that  $d_V(x_V, \partial U)=O(1)$. 
The corollary follows from a the triangle inequality. 
\end{proof}

\subsection{Product regions} \label{Sec:Standard}
For every subsurface $W$ of $S$,  we have a projection map
\[
\phi_W \from \X \to \X(W),
\]
defined by \thmref{Consistency}. Namely since the projections of a 
point $x \in \X$ to subsurfaces of $S$ are consistent, the projections to 
subsurfaces of $W$ are also consistent and hence can be
realized by a point in $\X(W)$. For points $x,y \in \X$, we define
\[
d_{\X(W)}(x,y) = d_{\X(W)}\big(\phi_W(x), \phi_W(y) \big)
\]
The subsurface $W$ is allowed to be an annulus in which case, $\X(W) = \C(W)$.

For a curve system $\balpha$, let $\X_\balpha$ be the set of points in 
$x \in \X$ where $\balpha$ is a subset of the pants decomposition $P_x$ 
associated to $x$. Consider a point $x \in \X_\balpha$.
Since every curve in $x$ is disjoint from $\balpha$, the projection of $x$ 
to any subsurface intersecting $\balpha$ is distance at most $2$ from the the 
projection of $\balpha$. Therefore for a sufficiently large threshold $T$ and 
for any $x,y \in \X_\alpha$, the set $\calW_T(x, y)$ consists only of 
subsurfaces disjoint from $\balpha$, each of which is contained in some 
component $W$ of $S \setminus \balpha$. Therefore, the map
\[
\Phi 
\from \X_\balpha \to \prod_W \X(W)
\qquad\text{where}\qquad
\Phi =\prod_W \phi_W
\]
is a quasi-isometry (the image is coarsely onto). Here the product space 
is equipped with the $L^1$--metric. A version of this theorem for
\Teich space was first proved by Minsky \cite{minsky:PR} and is
known as the product regions theorem. We see that the fact that $\Phi$ is 
a quasi-isometry is essentially immediate from the distance formula. But the 
proof of the distance formula in \cite{rafi:CM} used Minsky's product regions 
theorem. 

There is also a projection map $\phi_{\X_\balpha} \from \X \to \X_\balpha$; 
choosing a point in each $\X(W)$ for $W$ a component of $S\setminus\alpha$ and taking a union results in a point
in $\X_\balpha$. We define $d_{\X_\alpha}$ to mean the distance 
between projections to $\X_\balpha$. That is, for $x,y \in \X$, we define
\[
d_{\X_\balpha}(x,y) := 
d_\X \big(\phi_{\X_\balpha}(x), \phi_{\X_\balpha} (y) \big). 
\]
Note that the projection of $\phi_{\X_\balpha}(x)$ to $\X(W)$
is close to $\phi_W(x)$ because $\phi_W$ was defined using the consistency
result. Therefore, 
\begin{equation} \label{Eq:Xa}
d_{\X_\balpha} (x, y) \asymp 
\sum_W  d_{\X(W)} (x, y),
\end{equation}
where the sum is over components of $S \setminus \balpha$. 

We finish with an estimate of $d_\X$ using the projection distances
$d_{\X_\balpha}$. 

\begin{lemma}
\label{Lem:SumDistance}
Suppose $x,y\in\X$ and $\alpha_1,\ldots \alpha_k$ is a geodesic in $\calC(S)$  
joining $\alpha_1 \in P_x$ to $\alpha_k \in P_y$. Then 
\[
d_\X (x, y) \lmul \sum_{i=1}^k d_{\X_{\alpha_i}}(x,y).
\]
\end{lemma}
\begin{proof}
Let $M_0$ be the constant in \thmref{BGIT}. Then any subsurface $U$ with 
$d_U(x,y)\geq M_0$,  is disjoint  from some $\alpha_j$ and so is 
a subset of $S\setminus \alpha_j$.  Thus it appears as a term in 
some $d_{\X(W)}(x,y)$, where $W$ is a component of 
$S\setminus \alpha_j$. We are done by \eqnref{Xa}. 
\end{proof}

\section{Efficient paths are nearly geodesics}
\label{Sec:Efficient}
In this section we show that efficient paths are nearly geodesics
in the space $\X$. To do this, we use the construction of
Bestvina-Bromberg-Fujiwara \cite{bestvina:GQ} which gives a quasi-isometric
embedding of the mapping class group to a finite product of Gromov hyperbolic 
spaces. Their construction is completely axiomatic and works, essentially without 
modifications, for any of our spaces $\X$. We first review their construction. 

\subsection{A quasi-tree of curve complexes} \label{Sec:Y}
We summarize some statements in \cite{bestvina:GQ}.
Fix a threshold $K$.  Let $\bY$ be a collection of subsurfaces of $S$ with the 
property that if  $V,W \in \bY$ then 
\begin{itemize}
\item $V\pitchfork W$, and
\item every curve in $W$ intersects $V$. 
Hence, the domain of $\pi_{W,V}$ is $\C(W)$. 
\end{itemize}
Define $\bP_K(\bY)$ to be 
a graph whose vertices are elements of $\bY$ and two vertices $V,W$ are 
connected with an edge if for all $U\in \bY$
\[d_U(V,W)\leq K.\]
In \cite{bestvina:GQ} it is shown that $\bP_K(\bY)$ is quasi-isometric to a tree. 

Define $\C(\bY)$ to be  the space obtained from $\bP_K(\bY)$ by
attaching a copy of $\C(W)$ for every vertex $W\in \bY$ as follows:
The vertex set of $\C(\bY)$ is the union of vertex sets of $\C(W)$, 
$W \in \bY$. If $V$ and $W$ are joined by an edge in $\bP_K(\bY)$, 
then we join the vertex $\pi_V (\partial W)$ in $\calC(V)$ to 
$\pi_W (\partial V)$ in $\calC(W)$. It follows from \cite[Theorem E]{bestvina:GQ},
using the fact that each space $\C(W)$ is hyperbolic, that the resulting space 
$\calC(\bY)$ is also hyperbolic. 

Furthermore, in \cite{bestvina:GQ}, the authors show that when $S$
is connected:
\begin{itemize}
\item  the subsurfaces of $S$ 
can be decomposed into finitely many disjoint subsets $\bY^1, \ldots, \bY^k$ 
each having the transversality property mentioned above.
\item there is a finite index subgroup $\Gamma$ in $\Mod(S)$ which fixes each 
$\bY^i$
\end{itemize}
When $S$ is disconnected, we decompose the subsurfaces of each component
as above let $\bY^1, \ldots, \bY^k$ be the list of all such collections. 

We can assume $\bY^j$ contain only essential subsurfaces, that is, the thrice 
punctured spheres are always excluded, and in the case $\X$ is the pants 
graph, annuli are also excluded. Let 
\[
\bC=\prod_{j=1}^k \calC(\bY^j)
\] 
equipped with the $L^1$--metric. Thus  $\bC$ is a  product of finitely 
many hyperbolic spaces.

We define a projection map $\Psi^j \from \X \to \C(\bY^j)$
as follows:  For $x \in \X$, choose a subsurface $W_j \in \bY^j$ that minimizes 
\[ 
\max _{\alpha \in P_x} \, i(\alpha,\partial W_j),
\]
where the minimum is over $W_j\in \bY^j$. Define
\[
\Psi^j(x) = x_{W_j}.
\]
Recall that $x_{W_j}$ is the projection of $x$ to the curve complex
$\C(W_j)$. That is, $x_{W_j}$ is a point in $\C(W_j)$ and hence is a
point in $\C(\bY^j)$. 

\begin{remark} \label{Rem:Difference}
Our definition of the projection is slightly different from that of \cite{bestvina:GQ}.
They use the action of the mapping class group to define the projection. However,
in case $\bY^j$ consists of annuli and $\X$ is the augmented marking space,
$x_{W_j}$ is a point in the horoball $H\subset \HH^2$ and not a curve. 
In particular,  the action of the mapping class group is not coarsely  transitive. 
But in the other two cases, the two definitions match. 
\end{remark}

We claim that the consistency condition (\secref{Consistency}) 
shows that this map is coarsely well-defined; that is, up to a bounded distance 
in $\bC$,  the image is independent of the choice of $W_j$. 

First we note that since the finite index subgroup $\Gamma$ acts preserving 
each $\C(\bY^j)$ the above minimum is uniformly bounded by a constant 
independent of $x$.  We need to check that distinct choices $W_j$ and $V_j$ 
give points at bounded distance in $\C(\bY^j)$. First, we check,
for all $U \in \bY^j$, that  
\[
d_U(W_j,V_j)=O(1).
\] 
Choose an $\alpha \in x$ that intersect $U$. 
Since $i(\alpha,W_j)=O(1)$,  we have 
\[
d_U(\alpha , W_j)=O(1)
\] 
and similarly for $\alpha$ and $V_j$. The claim now follows from the 
triangle inequality. Therefore, $W_j$ and $V_j$ are connected by an edge
in $\bP_K(\bY^j)$.  To show that $x_{W_j}$ and $x_{V_j}$ are close 
in $\C(\bY^j)$, we need that 
\[
d_{V_j}(x_{V_j}, \partial W_j)=O(1)
\quad\text{and}\quad
d_{W_j}(x_{W_j}, \partial V_j)=O(1).
\]
This holds since for every curve $\alpha \in x$, 
\[
i(\alpha,\partial W_j)=O(1)
\quad\text{and}\quad
i(\alpha,\partial V_j)=O(1).
\]

We often denote $\Psi^j(x)$ by $x_j$. 
Now define a map
\[
\Psi \from \X \to \bC
\qquad\text{with} \qquad
\Psi = \prod_j \Psi^j. 
\]
  
The following theorem is proven in \cite{bestvina:GQ} for the mapping class group. 

\begin{lemma} \label{Lem:BBF}
There is $K' >K$ so that, for every $x, y \in \X$, and 
\[
d_\bC \big(\Psi(x), \Psi(y)\big)\geq 
\frac{1}{2}\sum_{W\in \calW_{K'}(x,y)} d_W(x,y),
\]
\end{lemma}
The proof of this lemma in \cite{bestvina:GQ} uses only the hyperbolicity of 
each curve complex and the consistency condition detailed in \secref{Consistency}
and works verbatim in our case. Hence we omit the proof. As a consequence
we have the following theorem which is also proven in \cite{bestvina:GQ}
for the mapping class group. We give a proof here because our
projection maps $\psi^j$ are defined differently from \cite{bestvina:GQ}. 

\begin{theorem} \label{Thm:Embedding}
The map $\Psi$ is a quasi-isometric embedding from $\X$ into $\bC$.
\end{theorem}

\begin{proof}
For $x \in \X$, the map $\Psi$ is defined by 
\[
\Psi(x)= \big( x_{W_1}, \ldots,  x_{W_k} \big),
\]
for some subsurface $W_j \in \bY^j$, $1\leq j \leq k$. 
We show $\Psi^j$ is quasi-Lipschitz which essentially the same 
as the proof the $\Psi^j$ is coarsely well defined. 
If $d_\X(x,x') = O(1)$, then 
\[
d_U(x,x')=O(1) \qquad \forall U.
\]
We have $i(x, W_j)=O(1)$, and so 
\[
d_U(x, W_j)=O(1)
\quad\text{and similarly}\quad
d_U(x',W_j')=O(1).
\]  
Together this gives   
\[
d_U(W_j,W_j') = O(1),
\]
which implies that $W_j$ and $W_j'$ are connected by an edge in
$\bP_K(\bY^j)$. We also know that $x$ is close to $x'$ which has
bounded intersection with $\partial W_j'$.
Hence
\[
d_{W_j}(x_{W_j}, \partial W_j') = O(1),
\quad\text{and similarly}\quad
d_{W_j'}(x_{W_j'}, \partial W_j) = O(1).
\]
Therefore, 
\[
d_{\C(\bY^j)}\big(\Psi^j(x),\Psi^j(x')\big)=O(1).
\]
This means, the maps $\Psi^j$ are quasi-Lipschitz, and so is $\Psi$. 

We need to find a lower bound for the distance between $\Psi(x)$
and $\Psi(y)$. By \lemref{BBF}, there is $K' >K$ so that, 
\[
d_\bC\big(\Psi(x), \Psi(y)\big)\geq 
\frac{1}{2}\sum_{W\in \calW_{K'}(x,y)} d_W(x,y),
\]
and since the distance formula works for any threshold,
\[
d_\X(x,y) \emul \sum_{W\in \calW_{K'}(x,y)} d_W(x,y).
\]
Hence
\[
d_\bC\big(\Psi(x), \Psi(y)\big) \gmul d_\X(x,y).
\]
This finishes the proof. 
\end{proof}

\subsection{Preferred paths and efficient paths}
\label{Sec:Preferred}
Since the space $\X$ is not hyperbolic, a quasi-geodesic connecting
two point in $\X$ may not be well behaved. Instead, we define
a notion of preferred path connecting two points in $\X$.

\begin{definition} \label{Def:Preferred}
Given $x,y \in \X$, we say a quasi-geodesic $\omega \from [a,b] \to \X$ is 
a \emph{preferred path} connecting $x$ to $y$ if, 
\begin{itemize}
\item $\omega(a) = x$, 
$\omega(b) = y$
\item  for every subsurface $U$, the map 
\[
\omega_U = \pi_U \circ \omega \from [a,b] \to \calC(U)
\]
is an unparametrized quasi-geodesic.
\end{itemize}
\end{definition}

\begin{lemma}
For any $x,y \in \X$ there is a preferred path connecting $x$ to $y$.
\end{lemma}
\begin{proof}
In the case of the mapping class group and the pants complex a resolution of a 
hierarchy is a preferred path \cite{minsky:CCII}. In the case that  $\X$ is 
\Teich space with the \Teich metric, such a path is constructed
in \cite[Theorem 5.7]{rafi:CM}. 
\end{proof}
 
\begin{remark}
It is known that the image of a \Teich geodesic is not a always preferred path 
(there may be backtracking in annuli).  It is unknown if the image of 
Weil-Petersson geodesics or a geodesic in the mapping class group is a
preferred paths.
\end{remark}

Now let $U$ be a subsurface of $S$ and 
let $x_U$ and $y_U$ be the projections of $x$ and $y$ to $\C(U)$. 
Denote a geodesic segment in $\calC(U)$ connecting $x_U$ to 
$y_U$ by $[x, y]_U$. Given $\kappa>0$, let 
\[
\calG(x,y,\kappa)= \Big\{ z \in \X \ST 
 \forall U, \,d_U\big(z, [x,y]_U \big)\leq \kappa \Big\}.
\] 

This notion was introduced in \cite{minsky:MR} where they call it the hull.
In a sense, this set if the union of all points in all preferred paths. 

\begin{lemma} 
There is a constant $\kappa_0$ depending only on the topology of $S$ and the 
constant involved in the definition of a preferred path so that, for any preferred 
path $\omega \from [a,b] \to \X$ and any $a\leq t \leq b$, 
$$
\omega(t) \in \calG(x,y,\kappa_0).
$$
\end{lemma}

\begin{proof}
Since $\C(U)$ is Gromov hyperbolic and the projection of $\omega$
is an unparametrized quasi-geodesic, it stays in a uniform neighborhood
of the geodesic connecting its end points.
\end{proof}

When $\kappa_0$ is fixed, we drop $\kappa_0$ and denote this set by 
$\calG(x,y)$.  
\begin{lemma} \label{Lem:G}
Let $x,y \in \X$  and $w,z \in \calG(x,y)$. Then 
\begin{itemize}
\item For any subsurface $U$, 
\[
d_U(w,z) \ladd d_U(x,y).
\]
In fact, $[w,z]_U$ is contained in a uniform neighborhood of $[x,y]_U$. 
\item (convexity) if $w,z \in \calG(x,y)$ then $\calG(w,z)$ is contained
in a uniform neighborhood of $\calG(x,y)$. 
\end{itemize}
\end{lemma}
\begin{proof}
Since $\C(U)$ is Gromox hyperbolic, if both $z_U$ and $w_U$ are  
close to $[x,y]_U$ so is $[z,w]_U$. Hence the length of 
$[z,w]_U$ is less than $[x,y]_U$ and any point close to 
$[z,w]_U$ is also close to $[x,y]_U$
\end{proof}

The following is the main theorem of this subsection and states
that efficient paths fellow travel preferred paths. 

\begin{theorem}
\label{Thm:Black-Box}
Let $\gamma\from [0,R] \to \X$ be an $\ep$--efficient path connecting 
$x=\gamma(0)$ to $y=\gamma(R)$. Then, the image of $\gamma$ stays 
in an $O(\ep R)$--neighborhood of $\calG(x,y)$. Moreover, it stays in an 
$O(\epsilon R)$--neighborhood of a preferred path 
connecting $x$ to $y$.
\end{theorem}

Before we prove this theorem, we need a few technical lemmas.
Given $x,y,z\in \X$ and a subsurface $W$, let $\eta_W$ be the center
of the triangle $(x_W, y_W, z_W)$ guaranteed by the hyperbolicity
of $\C(W)$. That is, $\eta_W$ is $\delta_W$--close to all three geodesics
$[x,y]_W$, $[y,z]_W$ and $[x,z]_W$, where $\delta_W$ is the hyperbolicity
constant of $\C(W)$. 

\begin{lemma}
The set $\{\eta_W\}$ is $O(1)$--consistent.
\end{lemma}

\begin{proof}
Let  $U,V$ arbitrary domains which are not disjoint.  We can assume that  
$U\not\subset V$ and hence $\pi_U(\partial V)$ is defined.  We know that the 
projections of $x,y,z$ to $U$ and $V$ are themselves $M_1$--consistent.

Consider the triangle $\Delta_U$ with vertices $x_U,y_U,z_U$ in $\calC(U)$. 
Note that, if $\partial V$ is uniformly close to all three edges of $\Delta_U$, then
$\partial V$ is uniformly close to $\eta_U$ and we are done. Hence, we can,
without loss of generality, assume that no point in $[x,y]_U$ is near $\partial V$
in $\C(U)$. In fact, since $\eta_U$ is $\delta$--close to $[x,y]_U$, we can assume 
every point in the convex hall of $x_U, y_U$ and $\eta_U$ is more than 
$M_1$ away from $\partial V$. 

This implies that $x_U, y_U$ and $\eta_U$ have defined projections
to $V$ and \thmref{BGIT} implies that their projections are a bounded 
distance from one another. That is, 
\[
d_V(\eta_U, x_U) = O(1), \quad 
d_V(\eta_U, y_U) = O(1),
\quad\text{and}\quad
d_V(x_U, y_U)=O(1). 
\]
On the other hand, because $d_U(x_U,\partial V)$ and $d_U(y_U,\partial V)$ 
are both larger than $M_1$, by \corref{Far} 
\[
d_V(x_V, x_U) \leq M_1 \quad\text{and}\quad
d_V(y_V, y_U) \leq M_1.
\]
And by triangle inequality
\[
d_V(x_V, y_V)=O(1). 
\]
From the definition of $\eta_V$, we have
\[
d_V(\eta_V, [x,y]_V) \leq \delta_V. 
\]
Again, using the triangle inequality, we get
\begin{equation} \label{Eq:eta}
d_V(\eta_U, \eta_V)= O(1). 
\end{equation}
This is the consistency condition when $V \subset U$. 

Thus assume $U\pitchfork V$. Since $\eta_U$ and $\partial U$ are disjoint, $d_V(\eta_U,\partial U)=O(1)$.  
This and \eqnref{eta} imply that
\[
d_V(\partial U, \eta_V)= O(1),
\]
which is the required consistency condition in this case. 
\end{proof}

Since the tuple of centers is consistent, it has a realization $\eta$. 
We call $\eta$ the center of the triangle $x$, $y$ and $z$. 

\begin{lemma} \label{Lem:Center}
For any $x,y,z \in \X$, let $\eta$ be the center of the triangle $\Delta$ with vertices  
$(x,y,z)$. Let $x_j, y_j, z_j$ and $\eta_j$ be projections of 
$x,y,z$ and $\eta$ to $\C(\bY^j)$. Then $\eta_j$ is near
the center of the triangle $(x_j, y_j, z_j)$ in $\C(\bY^j)$.
\end{lemma}

\begin{proof}
First we claim that, for every $W\in \bY^j$, $d_W(x, x_j)=O(1)$. 
Let $x_j$ be a curve $x_{V}$ in a surface $V\in \bY^j$. 
\begin{align*}
\text{$x_V$ is disjoint from $\partial V$}
  &\quad\Longrightarrow\quad d_W(\partial V, x_V) =O(1).\\
\intertext{and}
 \I(x, \partial V)=O(1) 
 &\quad\Longrightarrow\quad d_W(\partial V, x) =O(1).
\end{align*}  
The claim follows from the triangle inequality. 

In \cite[Lemma 3.13]{bestvina:GQ} it is shown that a geodesic
in $\C(\bY^j)$ connecting $x_j$ to $y_j$ is a bounded Hausdorff distance 
from a union of geodesics $[x_j, y_j]_W$, where the union is over
the subsurface where $d_W(x_j, y_j)$ is large. The same holds
for the geodesic connecting $x_j$ to $\eta_j$. 

But, as a consequence of above claim, the geodesic 
$[x_j, y_j]_W$ is a bounded Hausdorff distance from $[x,y]_W$ and
$[x_j, \eta_j]_W$ is a bounded Hausdorff distance from $[x,\eta]_W$.
Also, by assumption, we know that $[x,\eta]_W$ is contained in a
bounded neighborhood of $[x,y]_W$. Therefore, 
$[x_j,\eta_j]_W$ is contained in a bounded neighborhood of $[x_j,y_j]_W$.
In particular, if $d_W(x_j, \eta_j)$ is large, so is $d_W(x_j, y_j)$. 

That is, every subsurface that appears in the geodesic connecting
$x_j$ to $\eta_j$ also appears in the geodesic connecting $x_j$ to
$y_j$ and the portion of the geodesic $[x_j, \eta_j]$
that is in $W$ stays near the geodesic $[x_j, y_j]$.
Thus, $\eta_j$ is itself close to the $[x_j, y_j]$.

The same hold for $x_j, z_j$ and $y_j,z_j$. Since, $\C(\bY^j)$ is Gromov 
hyperbolic, and $\eta_j$ is close to all three geodesics, it is near the center
of the triangle. 
\end{proof}

We now prove the theorem. 

\begin{proof}[Proof of \thmref{Black-Box}]  
Let $\gamma_j= \Psi^j\circ \gamma$ be the projection of the path $\gamma$ to 
$\C(\bY^j)$. By \lemref{Easy}, each $\gamma_j$ is still $\ep$--efficient. 
Since $\C(\bY^j)$ is hyperbolic, by \lemref{Close-In-Hyp}, $\gamma_j(t)$ 
is within $O(\ep R)$ distance of the geodesic $[x_j, y_j]$. 
Let $z=\gamma(t)$ and let $\eta$ be the center of $x,y,z$.
From the construction, we have $\eta \in \calG(x,y)$. We estimate
the distance between $\eta$ and $z$. 

By \lemref{Center}, $\eta_j$ is the center of $x_j, y_j, z_j$ The distance from 
$z_j$ to $[x_j, y_j]$ is, up to an additive error, the distance from $z_j$ to the 
center $\eta_j$. Therefore, 
\[
d_{\C(\bY^j)}(z_j, \eta_j) \ladd 
   d_{\C(\bY^j)}\big(z_j, [x_j, y_j]\big) = O(\ep R). 
\]
It follows, since $\Psi$ is coarsely Lipschitz and the metric
in $\bC$ is the $L^1$--metric, that 
\[
d_\X(z, \eta) \lmul \sum_j d_{\C(\bY^j)}(z_j, \eta_j) = O(\ep R). 
\]
This finishes the proof of the first statement of the Theorem. 

We prove the second statement,  namely that, $\gamma$ stays in an 
$O(\ep R)$--neighborhood of a preferred path connecting x to y.
Let $\eta(t)$ be the center of $x, y$ and $\gamma(t)$. The issue is that 
$\eta(t)$ may not trace a preferred path since the $\ep$--efficient path $\gamma$ 
is allowed to backtrack up to $O(\epsilon R)$. We proceed therefore as follows. 

\begin{figure}[ht]
\setlength{\unitlength}{0.01\linewidth}
\begin{picture}(100, 25) 

\put(5,0){
\begin{tikzpicture} 
[thick, 
    scale=1,
    vertex/.style={circle,draw,fill=black,thick,
                   inner sep=0pt,minimum size= .5 mm}]

  \node[vertex] (x) at (.5,0)  [label=below:$x_W$] {}; 
    \node[vertex] (y) at (10.5,0) [label=below:$y_W$]  {}; 
  
 \node[vertex] (p1) at (8,.5) [label=right:$\gamma_W(s)$]  {}; 
 \node[vertex] (p2) at (6, 1.3) [label=right:$\gamma_W(t)$]  {}; 

 \node[vertex] (e1) at (8,0)  [label=below:$\omega_W(t)$] {}; 
  \node[vertex] (e2) at (6,0)  [label=below:$\eta_W(t)$]  {};

  \pgfsetplottension{0.75}
  \pgfplothandlercurveto
  \pgfplotstreamstart
  \pgfplotstreampoint{\pgfpoint{.5cm}{0cm}}  
  \pgfplotstreampoint{\pgfpoint{1cm}{.5cm}}   
  \pgfplotstreampoint{\pgfpoint{2cm}{.4cm}}
  \pgfplotstreampoint{\pgfpoint{2.5cm}{2.5cm}}
  \pgfplotstreampoint{\pgfpoint{3cm}{.8cm}}
  \pgfplotstreampoint{\pgfpoint{5cm}{.3cm}}
  \pgfplotstreampoint{\pgfpoint{8cm}{.5cm}} 
  \pgfplotstreampoint{\pgfpoint{5cm}{.7cm}} 
  \pgfplotstreampoint{\pgfpoint{6cm}{1.3cm}} 
  \pgfplotstreamend
  \pgfusepath{stroke}
 
 \draw (0,0) -- (11, 0);
 \draw[dashed] (p1) -- (e1);
 \draw[dashed] (p2) -- (e2); 
 
\end{tikzpicture}}
 
\put(80,17){$\C(W)$} 

\end{picture}
\caption{The point $\omega_W(t)$ is defined to be the point
 $\eta_W(s)$ that is farthest along in $[x,y]_W$, for $s \in [0,t]$.}
\label{Fig:omega} 
\end{figure}
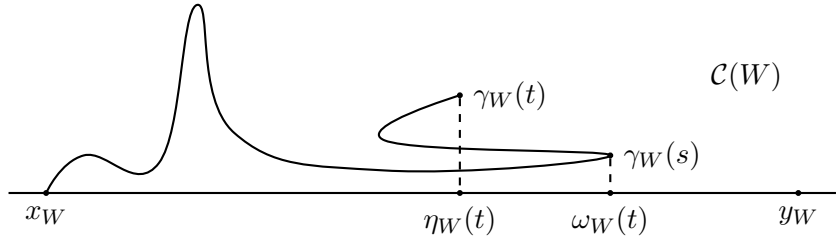

For a time $t$ and a subsurface $W$ let $\eta_W(t)$ be the projection of
$\eta(t)$ to $W$. Consider the geodesic $[x,y]_W$ in $\calC(W)$. 
Let $s \in [0,t]$ be a time where $d_W\big(x_W,\eta_W(s)\big)$ is maximized
(see \figref{omega}) and define 
\[
\omega_W(t)=\eta_W(s).
\]
Note that $\omega_W(t)$ is an unparametrized quasi-geodesic since
it stays close to the geodesic $[x,y]_W$ and does not backtrack. 

We prove $\{\omega_W(t)\}$ is  consistent. Pick two intersecting
surfaces $U$ and $V$. Suppose first that $U\pitchfork V$ and, without loss 
of generality,  $d_V(\partial U,y)=O(1)$. 
Assuming 
\begin{equation} \label{Eq:Large}
d_U\big(\omega_U(t),\partial V\big) \qquad \text{is large},
\end{equation}
we need to show $d_V\big(\omega_V(t),\partial U\big)=O(1)$. 
Since $\omega_U(t)=\eta_U(s)$ and  $\{\eta_W(s)\}$ is consistent, 
\eqnref{Large} implies
\begin{equation*}
d_V\big(\eta(s), \partial U\big)=O(1).
\end{equation*}  
Now $\omega_V(t)$ comes after $\eta_V(s)$ along $[x,y]_V$ 
and hence is close to the geodesic $[\eta(s),  y]_V$.
That is, the projections of $\eta(s)$, $y$, $\partial U$ and $\omega_V(t)$ 
to $V$ are all close to each other. In particular, 
\begin{equation*}
d_V\big(\omega_V(t), \partial U\big)=O(1).
\end{equation*}

Next assume $V\subsetneq U$ and  $d_U(\omega_U(t), \partial V)$ is large, 
or equivalently  $d_U(\eta(s),\partial V)$ is large. 
In $\C(U)$, one of the segments $[x, \omega(t)]_U$ or
$[\omega(t), y]_U$ is far from $\partial V$. Assume first that 
$[x_U,\omega_U(t)]_U$ is far from $\partial V$.  This implies that 
\[
d_V\big(x_U,\omega_U(t)\big)=O(1).
\] 
Now consider $\omega_V(t)=\eta_V(s')$ for some $s'\leq t$. Since $\eta_U(s')$ 
comes before $\omega_U(t)=\eta_U(s)$, we have that $d_U(\eta_U(s'),\partial V)$
is large. By Consistency we have 
\[d_V\big(\eta_U(s'),\omega_V(t)\big)=O(1)\] 
and since $\eta_U(s')$ comes between $x_U$ and $\omega_U(t)$ we also have 
\[
d_V\big(\eta_U(s'), \omega_U(t)\big)=O(1).
\] 
From the triangle inequality we get 
\[
d_V\big(\omega_U(t),\omega_V(t)\big)=O(1),
\] 
which proves Consistency in this case.

The remaining case is when $[\omega(t),y]_U$ is far from $\partial V$.
Consistency implies 
\[
d_V\big(\omega_U(t),\eta_V(s)\big)=O(1),
\]  
where $\omega_U(t)=\eta_U(s)$. We also have  
\[
d_V\big(y_U,\omega_U(t)\big)=O(1)
\quad\text{and so}\quad
d_V\big(y_U,\eta_V(s)\big)=O(1).
\]
Now $\omega_V(t)$ is further along the geodesic $[x,y]_V$ than $\eta_V(s)$ 
and so 
\[
d_V\big(y_U, \omega_V(t)\big)=O(1),
\] 
and applying the triangle inequality once more we have 
\[
d_V\big(\omega_V(t),\omega_U(t)\big)=O(1),
\] 
the desired inequality for Consistency.

Finally, we need to prove 
\[
d_{\X}\big(\gamma(t),\omega(t)\big)=O(\epsilon R).
\]
Let $\omega_j(t)$ be the projection $\omega(t)$ to $\C(\bY^j)$. 
We observe that $\omega_j(t)$ is near the point 
$\eta_j(s)$, $s \in [0,t]$, that is farthest along in $[x_j, y_j]$. 
This is because a geodesic is $\C(\bY^j)$ is a union of
geodesics in subsurfaces $W_1, \ldots, W_k$ appearing in natural
order, and since all these subsurfaces intersect, if a point
$z_j$ is ahead $z_j'$ along $[x_j, y_j]$, then the projection of 
$z_j$ is ahead of the projection $z'_j$ in every subsurface. 
Therefore  there is $s \in [0,t]$ so that
\[
\omega_j(t) = \eta_j(s). 
\]

From part one of \lemref{Easy} we have
\begin{align*}
d_{\C(\bY^j)} \big(x_j, \gamma_j(s)\big) 
  + d_{\C(\bY^j)} \big(\gamma_j(s), \gamma_j(t)\big)
  &+  d_{\C(\bY^j)} \big(\gamma_j(t), y_j \big) \\
 &\leq d_{\C(\bY^j)} (x_j, y_j) +O(\ep R). 
\end{align*}
But  $\C(\bY^j)$ is Gromov hyperbolic, so projection to $[x_j, y_j]$
is distance decreasing. Furthermore 
$d_{\C(\bY^j)}\big(\eta_j,[x_j,y_j]\big)=O(1)$.  Hence
\begin{align*}
d_{\C(\bY^j)} \big(x_j, \eta_j(s)\big) 
  + d_{\C(\bY^j)} \big(\eta_j(s), \eta_j(t)\big)
  &+  d_{\C(\bY^j)} \big(\eta_j(t), y_j\big) \\
 &\leq d_{\C(\bY^j)} (x_j, y_j) + O(\ep R). 
\end{align*}
But we know $\eta_j(t)$ comes before $\eta_j(s)$. 
Therefore, 
\[
d_{\C(\bY^j)} \big(\eta_j(s), \eta_j(t)\big) = O(\ep R),
\]
and hence, 
\[
d_{\C(\bY^j)} \big(\omega_j(t), \eta_j(t)\big) = O(\ep R).
\]
Since this is true for every $j$, we also have
\[
d_\X \big(\omega(t), \eta(t)\big) = O(\ep R).
\]
This finishes the proof.
\end{proof}

\section{Behavior of Preferred paths}
\label{Sec:Preferredprop}
In this section, we analyze preferred paths more carefully to obtain
more control over their behavior. In \propref{SomeBodyClose} we show
(up taking a subsurface) that if a preferred path is making progress in a 
subsurface, it has to stay close to the set of points in $\X$ that contain the 
boundary of that subsurface in their pants decomposition. 
This is analogous to the main results in \cite{rafi:SC} for \Teich geodesics, 
At the end of the section we prove two fellow traveling results 
\propref{fellowtraveling1} and \propref{fellowtraveling2}
for preferred paths. We start by proving a few lemmas.

The following Lemma give a bound on the thickness of 
$\calG(x,y)$ in terms of projection distances $d_{\X(W)}(x,y)$. 

\begin{lemma}
\label{Lem:Hausdorff}
For any $D>0$, if $d_{\X(W)}(x,y)\leq D$ for every subsurface $W$, 
then any point  $z\in \calG(x,y)$ is within  distance $O(D)$ of any preferred 
path $\gamma$ joining $x, y$.  
\end{lemma}

\begin{proof}
Let $\gamma$ be a preferred path connecting $x$ to $y$ and
let $z \in \calG(x,y)$.
If $S$ is not connected then $S = W_1 \cup, \ldots  \cup W_k$ and, for any 
time $t$,
\begin{align*}
d_\X\big(z, \gamma(t)\big) 
  &= \sum_i d_{\X(W_i)}\big(z,\gamma(t)\big) \\
  &\leq  \sum_i d_{\X(W_i)}(x,y) = O(D). \tag{\lemref{G}}
\end{align*}

Assume $S$ is connected and let $\gamma(t)$ be a point  with  
\[
d_{\C(S)}\big(z, \gamma(t) \big)=O(1).
\]
Such a point exists because $d_S(z, [x,y]_S)=O(1)$ by the definition
$\calG(x,y)$ and $\gamma$ connects $x$ to $y$. 

Let $\alpha_1, \ldots, \alpha_k$ be the geodesic in the $\C(S)$ connecting
a pants curve in $z$ to a pants curve in $\gamma(t)$ with $k=O(1)$. 
By \lemref{G}, for every $\alpha_i$, we have
\[
d_{\X_{\alpha_i}}\big(z, \gamma(t)\big) 
  \lmul d_{\X_{\alpha_i}}(x_1, x_2) \leq D,
\]
And by \lemref{SumDistance}
\begin{equation*}
d_\X\big(z, \gamma(t)\big) 
\lmul \sum_i d_{\X_{\alpha_i}}\big(z, \gamma(t)\big) =O(D). \qedhere
\end{equation*}
\end{proof}

 In preparation for the next lemma we recall a result of Rafi-Schleimer. 
They  give the following definition. 
\begin{definition}
Given a pair of points $x,y\in \X(S)$, thresholds $T_1\geq T_0>0$, and
a subsurface $W$, a collection $\Omega$ of subsurfaces 
$W_i\subsetneq W$ is an antichain in $W$ for $x$ and $y$ if 
\begin{itemize}
\item if $W_i\in\Omega$ then $d_{\calC(W_i)}(x,y)\geq T_0$.
\item if $d_V(x,y)\geq T_1$, then $V\subset W_i$ where $W_i\in\Omega$
\item if $W_i,W_j\in\Omega$ then $W_i$ is not a proper subsurface of $W_j$.  
\end{itemize}
\end{definition}

The size of $\Omega$ is a lower bound for the distance in the curve complex. 

\begin{lemma}[\cite{rafi:CC}] \label{Lem:RS}
There is a constant $A=A(T_0,T_1)$ such that 
\[
|\Omega|\leq A \, d_{W}(x,y).
\]
\end{lemma}

In the next lemma we show that if $\gamma$ is moving in some
$\X(W)$ but it is not close to $\X_{\partial W}$ it is because it is
really moving in a subsurface of $W$. 

\begin{lemma} \label{Lem:Stay-Far}
For all sufficiently large $M,D$ and any subsurface 
$W$ of $S$, if  $\gamma \from [a,b] \to \X$ is a preferred path connecting 
$x$ to $y$ such that  for all $t$,
\[
d_{\X}\big(\gamma(t),\X_{\partial W} \big)\geq M
\quad\text{and}\quad 
d_{\X(W)}(x,y)\geq D,
\] 
then there is a proper subsurface $V \subsetneq W$ such that 
\[
d_{\X(V)}(x,y)\emul d_{\X(W)}(x,y).
\] 
\end{lemma}

\begin{proof}
Let $M_0$ be the constant for \thmref{BGIT}. We argue in two cases. 

\subsection*{Case 1}  Assume $d_{W}(x,y)\leq M_2$, for some
uniform constant $M_2=O(M_0)$. By the distance formula 
\[
d_{\X(W)}(x,y) \emul \sum_{U\in \calW_{M_2}(x,y)} d_U(x,y).
\]
Here $\calW_{M_2}(x,y)$ is the collection of subsurfaces $U \subset W$
where $d_U(x,y)\geq M_2$. Since  $d_{W}(x,y)< M_2$, $W$ itself is not in 
the sum.  Consider the anti-chain $\Omega$ in $W$ for $x$ and $y$. Then
\lemref{RS} applied with $T_0=T_1=M_2$ implies that 
$\Omega = \{ V_1, \ldots, V_k\}$ where $k =O(M_0)$.
Each subsurface in $\calW_{M_2}(x,y)$ is a subset of some $V_i$ and
the number of subsurfaces $V_i$ is uniformly bounded. 
Hence, for $V$ equal to some $V_i$, we have
\[
d_{\X(W)}(x,y) \emul \sum_{U\in \calV_{M_2}(x,y)} d_U(x,y),
\]
where $\calV_{M_2}(x,y)$ is the collection of subsurfaces $U \subset V$
where $d_U(x,y)\geq M_2$. That is
$$
d_{\X(V)} (x,y) \emul d_{\X(W)}(x,y). 
$$

\subsection*{Case 2} Assume $d_{W}(x,y)$ is large compared to $M_0$.
We argue this case cannot occur. Choose $z\in \gamma$ 
whose projection to $\C(W)$ is at the midpoint of the quasi-geodesic 
$\pi_W\circ \gamma$. From our assumption, we know that both  
\[
d_W(x,z)
\quad\text{and}\quad  
d_W(z,y)
\] 
are large compared to $M_0$. Let $w \in X_{\partial W}$
be the projection of $z$ to $X_{\partial W}$. Then, 
for all $U$ disjoint from $W$, 
\[
d_U(z, w)=O(1).
\]  
By the assumption of the Lemma $d_{\X}(z,w)$ is large
and by the distance formula, there is some $U$ 
where $d_U(z,w)$ is large. From the previous equation, we know that
$U$ has to intersect $W$. There are two cases. 

Consider first the  possibility that $W\subsetneq U$ and
$d_U(z,\partial W)$ large.  Since the quasi-geodesic 
$[x,y]_U$ in $\C(U)$ passes through $z_U$,
either $[x,z]_U$ or $[z,y]_U$ stays far from $\partial W$.
Then, by \thmref{BGIT}, either  
$d_W(x,z)\leq M_0$ or $d_W(y,z)\leq M_0$ which is a contradiction. 

Consider now the  possibility that $U\pitchfork W$ so $d_U(z,\partial W)$ 
is large. By the first consistency condition we have that $d_W(z,\partial U)$ is small. 
The assumption that $d_W(x,z)$ and $d_W(y,z)$ are large 
(and the triangle inequality) now implies that both $d_W(x,\partial U)$ and 
$d_W(y,\partial U)$ are large. Again, the first consistency condition implies 
$d_U(x,\partial W)$ and $d_U(y,\partial W)$ are small, so $d_U(x,y)$ is small 
by the triangle inequality.  This in turn implies  $d_U(x,z)$ and $d_U(y,z)$ 
are small, and using triangle inequality one more time, we conclude that
$d_U(z,\partial W)$ is small. This is a contradiction.  
\end{proof}

\begin{proposition} \label{Prop:SomeBodyClose}
There exists constant $D_0$ such that   given a subsurface $W\subset S$ 
and a preferred path $\gamma \from [a,b] \to \X$ connecting $x$ to $y$ 
where 
\[
D= d_{\X(W)} (x,y)\geq D_0,
\]
there is a subsurface $V \subset W$ and a sub-interval $[c,d] \subset [a,b]$ 
so that
\begin{itemize}  
\item  $d_{\X(V)}\big(\gamma(c), \gamma(d)\big) \emul D$
 \item for $t \in [c,d]$,  $d_\X \big( \gamma(t), \X_{\partial V}\big) = O(1)$.
\end{itemize}
\end{proposition}

\begin{proof} We use induction on complexity of subsurfaces.  If $\gamma$ does not come within $M$ 
of  $\X_{\partial W}$ then \lemref{Stay-Far} applies. 
Let $W' \subsetneq W$ be a subsurface such that 
$$
d_{\X(W')}(x,y)\emul D.
$$ 
Since $W'$ has lower complexity than $W$, \propref{SomeBodyClose} 
applies by induction. That is, there is a subsurface $V \subset W' \subset W$ 
with the desired properties. 

Thus assume $\gamma$ does in fact come within $M$ of $\X_{\partial W}$, 
and let $z_1$ and $z_2$ be points in $\gamma$ marking the first and the last 
times  $\gamma$ is within $M$ of $\X_{\partial W}$.
We have, either 
\[
d_{\X(W)}(x, z_1) \emul D, \quad 
d_{\X(W)}(z_1, z_2) \emul D, \quad\text{or}\quad 
d_{\X(W)}(z_2, y) \emul D.
\]
If $d_{\X(S_1)}(z_1, z_2) \emul D$, then we are done after taking 
$V=W$. In the other two cases, (say $d_{\X_{\partial W}}(x, z_1) \emul D$)
the path connecting $x$ and $z_1$ does not come close to 
$\X_{\partial W}$ but  travels large distance in $\X_W$. Hence, we again can
apply \lemref{Stay-Far} and induction to finish the proof.
\end{proof}

\subsection{Steady Progress}
Consider a preferred path that stays near the space $\X_{\partial W}$
for some subsurface $W$. Sometimes, it is desirable that $\gamma$
makes steady progress in the curve complex of $W$. We make
this notion precise:

\begin{definition} \label{Def:Steady}
Suppose  $\gamma \from [a,b]\to  \X_{\partial W}$ is a preferred path connecting 
$x=\gamma(a)$ and $y=\gamma(b)$. Let $L=d_{\X(W)}(x,y)$  and
let $a=t_0<t_1<t_2<t_3<t_4<t_5=b$ such that $y_i=\gamma(t_i)$ satisfy 
\[
d_{\X(W)}(y_i,y_{i+1})=L/5.
\]
For a constant $C_0$, we say $\gamma$ makes \emph{$C_0$--steady progress} 
in $W$ if, for $i=0, \ldots, 4$
$$
d_W(y_i, y_{i+1}) \geq C_0.
$$ 
\end{definition}

Note that if $\gamma$ makes $C_0$ steady progress for some $C_0$ then it makes $C_0'$ steady progress for $C_0'<C_0$.

The next lemma says that we can  find subsurfaces where there is
steady progress. The constant $D_0$ appears in 
\propref{SomeBodyClose}, \lemref{Robust} and \propref{fellowtraveling1}. 
This means that we choose $D_0$ large enough that all three statements
hold. In the lemma below, $D_0$ seems to depend on $C_0$ which is
an open variable. But, in fact, the value of $C_0$ is fixed in 
\propref{fellowtraveling1} and should be taught of as a fixed constant. 

\begin{lemma}
\label{Lem:Robust}
For every $C_0$, there is $D_0$ such that for any surface $W$, if  $\gamma$ 
is a preferred path joining $x,y \in \calN_{O(1)}(\X_{\partial W})$ and 
\[
D:=d_{\X(W)}(x,y)\geq D_0
\] 
then, for some interval $[c,d]\subset [a,b]$ and some subsurface $V \subset W$, 
we have
\begin{itemize}
\item  $d_{\X(V)}\big(\gamma(c) , \gamma(d) \big)\gmul D$.
\item $\gamma|_{[c,d]}$ makes $C_0$-steady progress in $V$. 
\end{itemize}
\end{lemma}

\begin{proof}
We first note that by definition, a quasi-geodesic in  a lowest complexity 
subsurface makes steady progress;   otherwise since it is a quasi-geodesic it 
would have to make progress in the curve complex of {\em some} proper 
subsurface, but there are none. The proof is now by induction on complexity. If  
$\gamma$ does not make steady progress in $W$, then for some $i$,
$d_W(y_i, y_{i+1})=O(1)$.  If so, we use the anti-chains (\lemref{RS}) 
and argue as in Case 1 in the proof of \lemref{Stay-Far} to conclude
that there exists a subsurface $V \subsetneq W$ where
\[
d_{\X(V) }(y_i,y_{i+1})\gmul D.
\] 
Here the implied constant only depends on $C_0$ and not on $D$.  
Now, an induction on the complexity of $W$ implies the lemma
(replace $\gamma$ with the preferred path connecting $y_i$ to $y_{i+1}$
and $W$ with $V$).  
\end{proof}

We now prove a pair of fellow traveling lemmas. 
The first states that, if the end points of two preferred paths are close
compared to their lengths, and the first one makes steady progress in $W$ then
the middle part of the second one also stays near $\X_{\partial W}$. 

\begin{proposition}
\label{Prop:fellowtraveling1}
There are constants $c_0, c_1, C_0,D_0$ with the following property. Suppose 
$\gamma$ is a preferred path joining $x, y \in \X_{\partial W}$ that makes 
$C_0$-steady progress in $W$, let $z$ be the midpoint of $\gamma$ and assume 
$D:=d_{\X(W)}(x,y)\geq D_0$.  Suppose  $\gamma'$ is a preferred 
path joining $x'$ and $y'$ with
\[
d_{\X}(x,x')\leq c_0D,
\quad\text{and}\quad
d_{\X}(y,y')\leq c_0D.
\] 
Then, there is a subsegment of $\gamma'$ with length comparable
to $D$ that stays in a bounded neighborhood of $\X_{\partial W}$.
In fact, for $z'$ on $\gamma'$, if
\[
d_\X(z',z) \leq c_1 D
\qquad\text{then}\qquad
d_\X(z', \X_{\partial W}) = O(1). 
\]
\end{proposition}

\begin{proof}
Let $x=y_0, \ldots, y_5=y$ be as in \defref{Steady}. Consider the projection 
$\pi_W(\gamma)$  to $\C(W)$. We have $d_W(x,y_1)\gmul C_0$. 
Even though the distance in $\X$ between $y_1$ and $x$ is $D/5$, there may 
be points much closer to $x$ in $\X$ whose projection to  $\C(W)$ is still near 
$\pi_W(y_1)$. However, this can not happen if  we travel a few steps towards 
$y$ along $[x,y]_W$. We make this precise. 

\subsection*{Claim} There is $\beta_1 \in [x,y]_W$ so that
\begin{itemize}
\item $d_W(y_1, \beta_1) \leq 2\delta$. 
\item for any $z_1 \in \X$, 
\[
d_W(z_1, \beta_1) \leq \delta 
\quad \Longrightarrow\quad
d_\X(x, z_1) \gmul D.
\]
\end{itemize}
We remind the reader that $\delta$ is the hyperbolicity constant for $\calC(W)$.
\begin{proof}[Proof of Claim]
We know $d_{\X(W)}(x, y_1) \gadd D/5$ and that the distance in $\X(W)$ is
the sum of subsurface projections to subsurfaces in $\calW_T(x,y_1)$. 
The boundary of any such subsurface is near a curve in $[x,y_1]_W$. 

Let $\beta_1$ be the curve along $[x,y]_W$ that is $2\delta$ away 
(towards $y$) from the projection of $y_1$ to $W$.  For $T$ larger than $M_0$, 
by \thmref{BGIT} the projection of $[\beta_1,y]_W$ to any subsurface 
$U \in \calW_T(x,y_1)$ has uniformly bounded diameter $M_0$ 
(every curve in $[\beta_1,y]_W$ intersects $U$). 
In fact, for $z_1 \in \X$, where $d_W(z_1, \beta_1)\leq \delta$, every
curve in the geodesic $[z_1,y]_W$ also intersects $U$. Hence, 
$$
d_U(x, z_1) \gadd d_U(x, y_1).
$$
Therefore, $d_{\X(W)}(x, z_1) \gmul d_{\X(W)}(x, y_1)= D/5$. 
\end{proof}

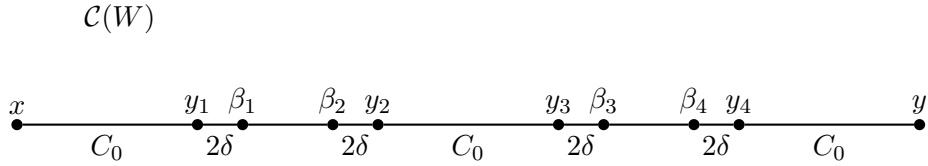
\begin{figure}[ht]
\setlength{\unitlength}{0.01\linewidth}
\begin{picture}(100, 18) 

\put(0,0){
\begin{tikzpicture}
   [thick, 
    scale=0.12,
    vertex/.style={circle,draw,fill=black,thick,
                   inner sep=0pt,minimum size=1.2mm}]

   \node[vertex] (y0) at (0,0)   {};
   \node[vertex] (y1) at (20,0)   {};
   \node[vertex] (y2) at (40,0)   {};
   \node[vertex] (y3) at (60,0)   {};
   \node[vertex] (y4) at (80,0)   {};
   \node[vertex] (y5) at (100,0)   {};
   \node [above] at (y0) {$x$};
   \node [above] at (y1) {$y_1$};
   \node [above] at (y2) {$y_2$};
   \node [above] at (y3) {$y_3$};
   \node [above] at (y4) {$y_4$};   
   \node [above] at (y5) {$y$};

   \node[vertex] (b1) at (25,0)   {};
   \node[vertex] (b2) at (35,0)   {};
   \node[vertex] (b3) at (65,0)   {};
   \node[vertex] (b4) at (75,0)   {};
   \node [above] at (b1) {$\beta_1$};
   \node [above] at (b2) {$\beta_2$};
   \node [above] at (b3) {$\beta_3$};
   \node [above] at (b4) {$\beta_4$};   

   \draw (y1) -- (b1) node[midway,below] {$2\delta$};
   \draw (b1) -- (b2);
   \draw (b2) -- (y2) node[midway,below] {$2\delta$}; 
   \draw (y0) -- (y1) node[midway,below] {$C_0$};
   \draw (y3) -- (b3) node[midway,below] {$2\delta$}; 
   \draw (y2) -- (y3) node[midway,below] {$C_0$};   
   \draw (b3) -- (b4);   
   \draw (b4) -- (y4) node[midway,below] {$2\delta$};         
   \draw (y4) -- (y5) node[midway,below] {$C_0$};        
\end{tikzpicture}}
 
\put(10,15){$\C(W)$} 
 
\end{picture}
\caption{The projections of $y_0, \ldots, y_5$ in $\C(W)$ are at least $C_0$ apart. 
For $i=1, \ldots 4$, the $\beta_i$  in $[x,y]_W$ is
$2\delta$ away from the projection of $y_i$ in the indicated direction.}
\label{Fig:beta_i} 
\end{figure}

Similarly we find
\begin{itemize}
\item a curve $\beta_4$ near the projection of $y_4$ so that, for any $z_4$
whose projection to $W$ is $\delta$ close to $\beta_4$, we have
$$
d_{\X(W)}(z_4, y) \gmul D.
$$
\item curves $\beta_2$ and $\beta_3$, near the shadows of $y_2$ and 
$y_3$ respectively, so that, for any $z_2,z_3 \in X$, 
\begin{equation} \label{Eq:Delta}
d_W(z_2, \beta_2) \leq \delta , \quad d_W(z_3, \beta_3) \leq \delta
\quad \Longrightarrow\quad
d_\X(z_2, z_3) \gmul D.
\end{equation}
\end{itemize}

If $c_0$ is small enough, any point in the path $[x,x'\,]_W$ has a distance 
of at least $C_0$ (up to an additive error) from any point in 
$[\beta_2, \beta_3]_W$. The same holds for curves in $[y,y'\,]_W$. 
Hence, if $C_0$ is much larger than the hyperbolicity
constant $\delta$, it follows from the hyperbolicity of the curve complex that
the $\delta$--neighborhood of the path $[x', y']_W$ has to contain 
$[\beta_2,\beta_3]_W$. In particular, 
$$
d_W(x', y') \gadd C_0.
$$
This means the path $[x', y']$ passes near $X_{\partial W}$. Let
$z_1'$ and $z_4'$ be the first and the last time the path $[x', y']$
is near $\X_{\partial W}$. A $\delta$--neighborhood of the 
geodesic $[z_1', z_4']_W$ must also contain $[\beta_2, \beta_3]_W$.

This  means that  there are points $z_2'$ and $z_3'$ along $[z_1',z_4']$ 
whose projection to $\C(W)$ is $\delta$--close to $\beta_2$ and $\beta_3$ 
respectively. Thus, by \eqnref{Delta}, 
$$
d_\X(z_2', z_3') \gmul D 
\quad \Longrightarrow\quad
d_\X(z_1', z_4') \gmul D.     
$$
This is the desired subsegment of $\gamma'$. To see the last assertion 
of the theorem, note that if $d_\X(z',z) \leq c_1 D$ for $c_1$ small
enough, then $z'$ is indeed in the segment $[z_2', z_3']$.
This finishes the proof. 
\end{proof}

\begin{proposition}
\label{Prop:fellowtraveling2}
Assume $S$ is connected. There  is a constant $C_1$ with the following property.  
For two pairs of points $x,y$ and $x',y'$, suppose   
\[
d_S(x,x')=O(\delta) 
\quad\text{and}\quad
d_S (y,y')=O(\delta).
\]
Suppose $z'\in \calG(x',y')$ is such that 
\[
d_S (z', x) \geq C_1\
\quad\text{and}\quad
d_S (z' ,y)\geq C_1.
\]
Then 
\[
z' \in \calG(x,y).
\]
\end{proposition}

\begin{proof}
We need to show that, for any subsurface $U$, 
\[
d_U \big(z', [x,y]_U\big)=O(1).
\]
If $U=S$ then we know from the hyperbolicity of $\C(S)$ that either $z_S'$ is 
close to $[x',x]_S$,  $[x,y]_S$ or $[y,y']_S$. Note that different
paths connecting two points in $\C(S)$ are a bounded distance apart in 
the $d_{\C(S)}$. Since,  $x_S'$ is much closer to $x_S$ than $z_S'$,  
then $z_S'$ is far from the path $[x,x']_S$ and similarly from $[y,y']_S$. 
Thus it has to be near $[x,y]_S$. 

Now assume $U \not = S$. If  $\partial U$ is not close to $z_S'$ in $\C(S)$ 
then, without loss of generality, we can assume $\partial U$ is far from 
$[z', y]_S$ (otherwise $\partial U$  would be  far from  $[x, z']_S$). 
Therefore, by \thmref{BGIT}  $d_U(z', y)=O(1)$. That is, $z_U$ is close to $y_U$
and hence close to $[x,y]_U$. 

If $\partial U$ is near $z_S'$ in $\C(S)$, then it is far from 
$[x.x']_S$ and $[y,y']_S$. Therefore 
\[
d_U(x,x') = O(1) 
\quad\text{and}\quad
d_U(y,y') = O(1).
\]
Hence, $[x,y]_U$ is near $[x', y']_U$. But we know that $z_U'$ is
near $[x', y']_U$ and therefore it is near $[x,y]_U$. This finishes the proof. 
\end{proof}

\section{Local structure of efficient maps} \label{Sec:Local}

We want to prove the following result. We assume $S$ is connected. 
\begin{theorem}
\label{Thm:Homework}
For all $R_0,\epsilon_0$ there is $R_1\geq R_0$ and $\ep_1<\ep_0$ 
so that if $B$ is a box in $\R^n$ with $|B|= R \geq R_1$ and 
$f\from B\to \X$ is an $\ep R$--efficient map with $\ep \leq \epsilon_1$, 
then there is a sub-box $B'\subset B$ with $R'=|B'| \geq R_0$
such that one of the following holds:
\begin{itemize}
\item $R'\gmul \sqrt[3]{\ep^2}\, R$ and, for some curve $\alpha$,  
$f(B')$ lies  within an $O(\sqrt[3]{\ep}\, R')$--neighborhood of $\X_\alpha$.
\item  $R'\gmul \sqrt[3]{\ep}\,R$ and there exist $x,y$ so that $f(B')$ lies 
within an $O(\sqrt[3]{\ep}\,R')$--neighborhood of a preferred path  in $\X$ 
joining $x$ to $y$.  
\end{itemize}  
\end{theorem}

\begin{proof}
We can assume that the diameter $f(R)$ is at least $\sqrt[3]{\epsilon} R$; 
otherwise the first case holds by taking $B'=B$, $R'=R$ and any curve
$\alpha \in \pi_S\circ f(B)$. Thus choose $x=f(p),y=f(q)$ where $p,q\in B$ 
such that 
\[
d_\X(x,y) \geq \sqrt[3]{\ep}\,R.
\]
Identify the geodesic segment $[p, q]$ with an interval
in $\R$. By \thmref{Black-Box}, $f\big([p,q]\big)$ stays in the
$O(\ep R)$--neighborhood of a preferred path $\gamma$ joining $x,y$. 

Now suppose, for some proper subsurface $W\subsetneq S$, that 
\[
d_{\X(W)}(x,y)\geq \sqrt[3]{\ep^2}\, R.
\]
We now claim the first conclusion of the Theorem holds. That is,  
there is a subbox $B'$ of size 
\[
R' \emul d_{\X(W)}(x,y) \gmul \sqrt[3]{\epsilon^2} R
\] 
and a curve $\alpha$ such that $f(B')$ lies  within an 
$O(\sqrt[3]{\epsilon} R')$--neighborhood of $\X_\alpha$.
We prove the claim. 

It follows from \propref{SomeBodyClose} and \lemref{Robust}
that we can find a subinterval $[d_1,d_2]\subset  [a,b]$, a subsurface $V$ and 
$C_0$ such that 
\begin{enumerate}
\item $d_{\X(V)}\big(\gamma(d_1) , \gamma(d_2)\big)\gmul   R'$,
\item $\gamma\big([d_1,d_2]\big)$  stays in a $O(1)$ neighborhood of 
$\X_{\partial V}$, and
\item the path $\gamma|_{[d_1,d_2]}$ makes $C_0$-steady progress in $V$. 
\end{enumerate}
Let $x_i=\gamma(d_i)$.  Let $s_i\in [p,q]$ so that  
\[
d_{\X}(f(s_i),x_i)=O(\epsilon R).
\]
For small  $c$ consider any two points 
$t_1,t_2$ at distance $cR'$ from $s_1,s_2$ respectively and set $y_i=f(t_i)$.  
 For $c$ small enough we have 
\begin{align*}
d_\X(x_i,y_i) &\leq d_\X \big(x_i, f(s_i)\big) + d_\X\big(f(s_i),y_i \big)\\
& \leq O(\ep R) + K cR' \leq c_0R'
\end{align*}
where $c_0$ is the constant given by \propref{fellowtraveling1}.  
Let  $d\in [d_1,d_2]$ be such that $x=\gamma(d)$ 
is the  midpoint of a $\gamma|_{[d_1,d_2]}$. Let $p$ be any point such that  
\[
d_{\X}(f(p),x)=O(\epsilon R).
\]
By  \propref{fellowtraveling1}, 
there is $c_1$ so that all $y\in\calG(y_1,y_2)$ that satisfy  
\[
d_\X(y,x) \leq c_1R'
\]
also satisfy
\[
d_\X(y,\X_{\partial V})=O(1).
\]
For $c$  small, a box of size $R'$ centered at $p$ is mapped 
under $f$ within distance $O(\epsilon R)=O( \sqrt[3]{\epsilon} \, R')$ of such 
$y$ and so  the image of the box lies within 
$O(\sqrt[3]{\epsilon}\, R')$ of $\X_\alpha$, for $\alpha= \partial V$. 
This proves the claim. 

We continue the proof of the Theorem. 
By the first part of the argument we can assume  that  for all $W\subsetneq S$ that 
\begin{equation}
\label{Eq:bounded3} 
d_{\X(W)}(x,y)\lmul  \sqrt[3]{\epsilon^2} \, R
\end{equation}
This and \lemref{SumDistance} imply that $\gamma$ makes $C_0'$-- steady 
progress in the entire surface $S$, for some $C_0' \gmul 1/ \sqrt[3]{\ep}$. 
For $\epsilon$ small enough, this implies that it makes $C_0$ steady progress 
where $C_0$  is the fixed constant of Proposition~\ref{Prop:fellowtraveling1}.
Let $C_1$ be the constant of \propref{fellowtraveling2}.
For a small but fixed $c>0$, take a $cR'$--neighborhood of $p$ and a 
$cR'$--neighborhood of $q$ where now 
\[
R'=d_{\X}(x,y)\gmul \sqrt[3]{\epsilon} \, R.
\]  

Let $p',q'$ be any points in these neighborhoods and let $x'=f(p')$ and $y'=f(q')$. 
By \thmref{Black-Box} we can find a preferred path  $\gamma'$ 
joining $x',y'$  within $O(\epsilon R)$ of $f\big([p',q']\big)$.  Since the map $f$ is 
quasi-Lipschitz it follows, for $c$ sufficiently small, that 
\[
d_{\X}(x',y')\gmul R'.
\]

Choose any  point $\hat p$ in the middle third of $[p',q']$. There is
$z'\in \calG(x',y')$ whose projection to $\calC(S)$ is at least $C_1$--far 
from $x_S'$ and $y_S'$ and
\[
d_{\X}(f(\hat p),z')=O(\epsilon R).
\]
By \propref{fellowtraveling2}, we know that $z'\in\calG(x,y)$  and so 
$f(\hat p)$ is within $O(\epsilon R)$ of $\calG(x,y)$ and by \lemref{Hausdorff} 
any point of $\calG(x,y)$ is within  distance  $O(\sqrt[3]{\epsilon^2}\, R)$ of 
$\gamma$.  

We have shown that any point in the middle third of 
any segment starting near $p$ and ending near $q$ is mapped to 
a point that is in a $O(\ep R)$--neighborhood of $\gamma$. 
But such a path covers a box of size $R'$.  Thus, there is box of size 
$R'$ which maps within $O(\sqrt[3]{\epsilon^2} \, R)=O(\sqrt[3]{\epsilon} \,R')$ 
of a preferred path.  We are done.
\end{proof}  

\section{Proof of main theorems} \label{Sec:Proof}

We are ready to prove \thmref{Intro-Standard-Flat} and \thmref{Rank}.
We first prove a version of \thmref{Intro-Standard-Flat} for efficient maps.
Then, we use coarse differentiation to finish the proof. 

\begin{definition} \label{Def:StandardFlat}
Let $\balpha$ be a (possibly empty) curve system. For every connected 
component $W$ of $S \setminus \balpha$ (including annuli if
$\X$ is not the $\calP(S)$), let $\omega_W \from I_W \to \X(W)$ be a 
preferred path. Consider the box $B = \prod_W I_W \subset \R^n$,
where $n$ is the number of components of $S \setminus \balpha$. Then
\[
F\from B \to \X_\balpha= \prod_W \X(W)
\qquad\text{where}\qquad
F= \prod_W \omega_W,
\]
is a quasi-isometric embedding because each $\gamma_W$ is a quasi-geodesic
and the product space is equipped with the $L^1$--metric. 
We call this map a \emph{standard flat}
in $\X$.
\end{definition}

\begin{theorem} \label{Thm:Standard-Flat} 
Let $S$ be a surface with complexity $\xi=\xi(S)$. For given
$\ep_0$ and $R_0$, let 
\[
\ep_\xi = \ep_0^{(6^\xi)} 
\qquad\text{and}\qquad
R_\xi = \frac{R_0}{\ep_\xi}.
\]
Assume $f \from B \to \X$ is an $\ep_\xi$--efficient map where $B$ is a 
box of size $R_\xi$ in $\R^n$. Then, there is a box $B' \subset B$ of size 
$R'\geq R_0$ so that the image $f(B')$ lies inside the 
$O(\ep_0 R')$--neighborhood of a standard flat in $\X$. 
\end{theorem} 

\begin{proof}
We prove the theorem by induction on the complexity $\xi=\xi(S)$
of the surface $S$ (see \eqnref{Complexity}). If $\xi=0$, then
\[
S = \coprod_{i=1}^m S_i,
\] 
where each $S_i$ is either a once-punctured torus or
a four-times-punctured spheres. When $\X$ is the pants complex,
$\X(S_i)$ is quasi-isometric to the Farey graph;
when $\X$ is the augmented marking space, $\X(S_i)$ is 
isometric to a copy of the hyperbolic plane; and when $\X$ is the 
marking complex $\X(S_i)$, is a graph whose vertices are the  edges of the 
Farey graph and two vertices are connected if the associated edges have a 
common vertex. The latter is 
known to be quasi-isometric to a tree. Hence, in all cases, $\X(S_i)$ is a 
Gromov hyperbolic space. That is, $\X$ is a product of hyperbolic spaces. 

In this case $R_\xi=R_0$ and $\ep_\xi=\ep_0$. Let $f_i \from B \to \X(S_i)$ 
be the projection of $f$ to $\X(S_i)$. by \lemref{Easy}, $f_i$ is still 
$\ep_0$--efficient. Applying, \propref{Efficient-to-Hyperbolic} to 
$f_1 \from B \to \X(S_1)$, we obtain a sub-box $B_1$ where $f_1(B_1)$ 
lies in an $O(\ep_0 R_0)$--neighborhood of a quasi-geodesic in $\X(S_1)$. 
But in this case, quasi-geodesics are also preferred paths. 
Now we apply \propref{Efficient-to-Hyperbolic} to 
$f_2 \from B_1 \to \X(S_2)$ to obtain a box $B_2$ so that $f_2(B_2)$ lies in 
an $O(\ep_0 R_0)$--neighborhood of a preferred path in $\X(S_2)$. Continuing 
this way, we find a box $B_m$ where the image of every $f_i$ lies in an 
$O(\ep_0 R_0)$--neighborhood of a preferred path in $\X(S_i)$. This means 
$f(B_m)$ lies in a $O(\ep_0 R_0)$--neighborhood of a standard flat in $\X$.  
Note that $B_m$ has the same size as $B$ (within uniform multiplicative error). 
This proves the base case of the induction. 

Assume now that $\xi$ is non-zero. Apply \thmref{Homework}. If the 
second condition holds, we are done for 
\[
R' \gmul  \sqrt[3]{\ep_\xi} \, R_\xi 
  = \frac{\sqrt[3]{\ep_\xi}}{\ep_\xi} \, R_0 \geq R_0,
\]
since a preferred path  is itself is a standard flat and 
$\sqrt[3]{\ep_\xi} = \ep_{\xi-1}^2 \leq \ep_0$.

Otherwise, we have a box $\bar B$ of size 
\[
\bar R \gmul \sqrt[3]{\ep_\xi^2}\,  R_\xi \geq R_{\xi-1}. 
\]
that maps to a $O(\sqrt[3]{\ep_\xi} \ \bar R)$--neighborhood of $\X_\alpha$ for
some curve $\alpha$. The map $f$ is $\ep_{\xi-1}^2$--efficient
because, reducing the size of the box by some factor (in this case
$ {\sqrt[3]{\ep_\xi^{-2}}}$) only makes the efficiency constant 
increase  by the same factor and
\[
\frac {\ep_\xi}{ \sqrt[3]{\ep_\xi^2}} = \sqrt[3]{\ep_\xi} = \ep_{\xi-1}^2.
\]

Composing $f$ with the closest point projection map to $\X_\alpha$ and using 
the fact that $\X_\alpha$ is quasi-isometric to 
$\C(\alpha) \times \X(S \setminus \alpha)$, we obtain a map
\[
\bar f \from  \bar B \to \C(\alpha) \times \X(S \setminus \alpha).
\]
By part (4) of \lemref{Easy}, $\bar f$ is $\ep_{\xi-1}$--efficient. 

Projecting to the second factor, we have a $\ep_{\xi-1}$--efficient map from a 
box of size $R_{\xi-1}$ to $\X(S \setminus \alpha)$
which by the inductive assumption has a sub-box $B_0$ of size at least $R_0$ that 
stays in $O(\ep_0 R_0)$--neighborhood of a standard flat in 
$\X(S\setminus \alpha)$. Now projecting to the first factor and applying
\propref{Efficient-to-Hyperbolic} ($\C(\alpha)$ is hyperbolic), we find a sub-box 
$B'$ of size at least $R_0$ that stays in a $O(\ep_0 R_0)$--neighborhood of a line
in $\C(\alpha)$. That is, $f(B')$ stays in a $O(\ep_0 R_0)$--neighborhood
of a standard flat in $\X$. This finishes the proof. 
\end{proof}

\begin{proof}[Proof of \thmref{Intro-Standard-Flat}]
Let $R_1$ be large enough that the box $B$ is guaranteed 
(by \thmref{Differentiable}) to have
a sub-box $B_\xi$ of size at least $R_\xi$ where the restriction of $f$
to $B_\xi$ is $\ep_\xi$--efficient. Note that we are not using 
the full force of the theorem; we need only one efficient sub-box.
Apply \thmref{Standard-Flat} to $f \from B_\xi \to \X$ to obtain the theorem. 
\end{proof}

\begin{proof}[Proof of \thmref{Rank}]
 Pick $1/\ep_0 \gg K$ and $R_0 \gg C/\ep_0$. Apply 
\thmref{Intro-Standard-Flat} to obtain a constant $R_1$
and let $R_2$ be a any constant greater than $R_1$. Then, the image
of a sub-box $B'$ of $B$ of size $R'>R_0$ is in a 
$O(\ep_0 R')$--neighborhood of a flat $F \from \R^m \to \X$. 
Taking a composition of $f$, the closest point projection to the image of $F$ 
and then $F^{-1}$, we obtain a map  $G \from B' \to \R^m$ with the 
property that, for $p, q \in B'$, 
\begin{equation} \label{Eq:NearlyQI}
d_{\R^m} \big( G(p), G(q) \big) \emul 
d_{\R^n} (p,q) \pm O(\ep R'). 
\end{equation}
We show that there is no such a map if $n$ is bigger than 
$\rank_{top}(\X)\geq m$. The proof is similar to the proof of nonexistence
of quasi-isometries between $R^n$ and $R^m$.  
Consider a net of $O(\ep_0 R')^n$ points in $B'$ that are pairwise
$K_1 \ep_0 R'$ apart, where $K_1$ is much larger than constants involved
in \eqnref{NearlyQI}. Then,  by the choice of $K_1$, the image points are at least 
distance $\epsilon_0 R'$--apart and are contained in a ball of radius $O(R')$ in
$\R^m$. The number points in a such net is of order of
$O(R' \ep_0)^m$.  Choosing $R'$ large enough (which can be done by choosing 
$R_0$ large) we obtain a contradiction. 

To see the second assertion, we note that if, for every subsurface $W_i$,
$\C(W_i)$ contains an infinite geodesic, then the product of these
geodesics is a quasi-isometric image of $R^n$. This fails when $W_i$
is an annulus and $\X$ is the augmented marking space (a horoball does not 
contain a bi-infinite geodesic). In this case, we choose a $\ep_0$--thick point 
$X$ in $\T(S)$ and a pants decomposition of curves of length  at most some fixed $B$.   
% and take $B>0$ large enough so that the set of curves of length 
%ess than $B$ fill $X$. This set contains many (but uniformly bounded number of) 
%pants decompositions. For any such pants decomposition $P$ (where 
%the length in $X$ of every curve in $P$ is less than $B$), 
The point $x \in \X$
associated to $X$ is uniformly close to the product region associated to $P$ 
(see \secref{Standard}). Consider an infinite ray for every horoball associated to a 
curve in $P$. The product of these rays is a quasi-isometric image of an orthant 
in $\R^n$. 
%The union of these is quasi-isometric to $\R^n$ because the total 
%number of orthant glued at $x$ is uniformly bounded.   
\end{proof}

\bibliographystyle{alpha}

\end{document}